\documentclass[12pt]{article}
\usepackage{rotating}
\usepackage[ruled]{algorithm}
\usepackage{color}
\usepackage{amssymb}
\usepackage{amsmath}
\usepackage{bm}
\usepackage{xcolor}
\usepackage{mdframed}
\usepackage{amsmath,amsthm}
\usepackage{amsfonts}
\usepackage{graphicx}
\usepackage{graphics}
\usepackage{setspace}
\usepackage{listings}
\usepackage{bm}
 
\numberwithin{equation}{section}
\newtheorem{Theorem}{Theorem}[section]

\newtheorem{lemma}[Theorem]{Lemma}

\newcommand{\ee}{\text{e}}

\DeclareMathAlphabet      {\mathbfit}{OML}{cmm}{b}{it}
\begin{document}
\date{\small\textsl{\today}}
\title{Analysis of a Legendre spectral element method (LSEM) for  the two-dimensional  system of a nonlinear stochastic advection-reaction-diffusion models}
\author{
 \small{Mostafa Abbaszadeh}$^{\mbox{\small \em a}}\vspace{.5cm}$,~\small{Mehdi Dehghan}$^{\mbox{\small \em a}}$\footnote{Corresponding author. {\em E-mail addresses:}
 	m.abbaszadeh@aut.ac.ir (M.Abbaszadeh), mdehghan@aut.ac.ir, mdehghan.aut@gmail.com (M. Dehghan),  amirreza.khodadadian@ifam.uni-hannover.de (A. Khodadadian), thomas.wick@ifam.uni-hannover.de (T.Wick). },~\small{Amirreza Khodadadian}$^{\mbox{\small \em b}}$,~
\small{Thomas Wick}$^{\mbox{\small \em b}}$  \\
\small{\em $^{\mbox{\footnotesize a}}$\em Department of Applied
Mathematics, Faculty of Mathematics
and Computer Sciences,}\vspace{-3mm}\\
\small{\em Amirkabir University of Technology, No. 424, Hafez
Ave., 15914, Tehran, Iran}\vspace{-0.5mm}\\
\small{\em $^{\mbox{\footnotesize b}}$\em Institute of Applied Mathematics, Leibniz University of Hannover }\vspace{-3mm} \\
\small{\textit{  Welfengarten 1, 30167 Hanover, Germany}}\\
}
\maketitle

\vspace{-1cm}
\noindent \noindent\linethickness{.5mm}{\line(1,0){540}}
\begin{abstract}
In this work, we develop a Legendre spectral element method  (LSEM)   for solving the   stochastic nonlinear system of advection-reaction-diffusion models. The used basis functions are based on a class of Legendre functions such that their mass and diffuse matrices are  tridiagonal  and diagonal, respectively. The temporal variable is discretized by a Crank--Nicolson finite difference formulation.  In the stochastic direction, we also employ a random variable $W$ based on the $Q-$Wiener process. We inspect  the rate of convergence and the unconditional stability for the achieved semi-discrete formulation.
 Then, the Legendre spectral element technique is used to obtain a full-discrete scheme. The error estimation of the proposed numerical scheme is substantiated based upon the energy method. The numerical results confirm the theoretical analysis.   
\vspace{.1cm}
\end{abstract}
\section{Introduction}
We consider the   stochastic nonlinear system of advection-reaction-diffusion models  \cite{Kobus,W.Liu}
\begin{equation}\label{1gw}
\left\{ \begin{array}{l}
\displaystyle du + \Big( \xi ({ \mathbfit{x}})\nabla u - \nabla \cdot {\zeta({ \mathbfit{x}})\nabla u}  + {w_p}{e_1}f(u,v)\Big)  dt = dW, \qquad \hspace{2.5cm} \text{in} ~D\times (0,T],\\
\\
\displaystyle dv + \Big(   \xi ({ \mathbfit{x}})\nabla v - \nabla \cdot\left( {\zeta({ \mathbfit{x}})\nabla v} \right) + {w_p}{e_2}f(u,v) \Big) dt = dW, \hspace{3.1cm} \text{in} ~D\times (0,T],\\
\\
\displaystyle dw + \Big(  \xi
({ \mathbfit{x}})\nabla w - \nabla \cdot\left( {\zeta({ \mathbfit{x}})\nabla w}
\right) + {w_p}{e_3}f(u,v) + r({ \mathbfit{x}})w \Big)  = dW,\hspace{1.6cm} \text{in} ~D\times (0,T],
\end{array} \right.
\end{equation}
where $u$, $v$ and $w$ denote  the concentrations of the main ground substance, aqueous solution electrolyte and microorganism, respectively \cite{Kobus,W.Liu}. In the above model $r(x)$ is a known function, $\xi$ is the advection coefficient, $\zeta$ is the diffusion coefficient,  $e_i$ and $w_p$ are constant, respectively.  Also, $W$ is a $Q$-Wiener process  with respect
to a filtered probability space $(\Omega,\mathcal{F},\mathbb{P})$.
The nonlinear terms are
\[f(u,v) = g(u,v) = h(u,v) = \frac{u}{{{\kappa _1} + u}} + \frac{v}{{{\kappa _2} + v}}.\]
Predictions of solute transport in aquifers generally have to rely on mathematical models  based on groundwater 
flow and convection-dispersion equations. 
The groundwater model is employed to prevent and control the groundwater contaminant with the microbiological technology
\cite{W.Liu}. 
 Several scholars investigated Eq.  \eqref{1gw} for example using an improved finite element approach \cite{W.Liu}, meshless local approaches \cite{ilati2016remediation,Sarra1},  lattice Boltzmann technique \cite{budinski2015lattice}, a front-tracking method \cite{d2007numerical},  novel WENO methods \cite{Jiang}, or a finite element method \cite{B.Liu}. The interested readers can refer to \cite{wang2016estimating,zhang2017study} to get more information for Eq. \eqref{1gw}.

 In the past, the groundwater models have been based only on deterministic considerations.
 In practice, aquifers are generally heterogeneous, i.e., their hydraulic 
 properties (e.g., permeability)  change in space.  These variations  are irregular and characterized by length  
 scales significantly larger than the pore scale. These spatial 
 fluctuations cause the flow variables  such as 
 concentration to change  in space in an irregular manner. Therefore, a reliable description of the groundwater model
 can be explained only in a stochastic form  \cite{unny1989stochastic}.

The first stochastic equation can be rewritten as
\begin{align}\label{sode}
d u(t)= \Big( A u(t)+f(u)\Big)  dt + dW,
\end{align}
where $-A:\mathcal{D}(-A)\subset H\to H$ is a linear, self-adjoint, positive definite operator where the domain $\mathcal{D}$ is dense in $H$ and compactly embedded in $H$ (i.e., $L^2(D)$)
 and the semigroup $\text{e}^{tA}$ $(t\geq0)$ is generated by $-A$. Additionally, we assume that $f:H\to H $ satisfies the linear growth condition and is twice continuously Frechet differentiable with bounded derivatives up to order 2 \cite{wang2013weak}. The initial value $u(0)=u_0$ is deterministic as well. Therefore, \eqref{sode} has a continuous mild solution \cite{haji2016multi}
\begin{align}
\label{mild}
u(t)=\text{e}^{tA} u_0+\int_{0}^{t} \text{e}^{(t-s)A}f(u(s))~\text{d}s+  \int_{0}^{t} \text{e}^{(t-s)A}~\text{d}W(s),
\end{align} 
where for $t\in [0,T]$ and $u:[0,T]\times D\rightarrow H$. Regarding the expected value of the solution, we can assume that  $\mathbb{E}\|u(t)\|^2 \leq \infty$. The same mild solutions can be employed for $v$ and $w$.

The deterministic case of Eq. \eqref{1gw} has been studied by some scholars for example a new finite volume method  \cite{W.Liu}, new Krylov  WENO methods \cite{Jiang}, local radial basis function collocation method \cite{Y.Hon}, etc. Also, the SEM is applied to solve some important problems such as the Schr\"{o}dinger equations \cite{F.Zeng5}, Pennes bioheat transfer model \cite{DehghanS}, the shallow water equations \cite{GIRALDO},  integral differential equations \cite{Fakhar1,Fakhar2,Fakhar3}, hyperbolic scalar equations \cite{Maerschalck}, predator-prey problem \cite{DehghanS1},  some problems in the finance mathematics \cite{Zhu1, Zhu2} and so forth.
 \\ \\
 The main aim of the current paper is to propose a new high-order numerical procedure for solving the   two-dimensional  system of a nonlinear stochastic advection-reaction-diffusion models. The used technique is based on the modified Legendre spectral element procedure. The coefficient matrix of the employed technique is more well-posed than the traditional Legendre spectral element method.  The structure  of this article is as follows. In Section \ref{section2}, we propose and analysis the time-discrete scheme. In Section \ref{section3}, we develop the new numerical technique and analysis it.  We check the
 numerical results to solve the considered model in Section \ref{section4}.  Finally, a brief conclusion of the current paper is
 written in Section \ref{Section5}.


\section{Temporal discretization}
\label{section2}
First of all, we briefly review some important notations used in the paper. Considering $\Omega\subset \mathbb{R}^d$, we define the following functional spaces
\[\begin{array}{l}
 {L^2}(\Omega ) = \left\{ {f:\,\,\,\,\,\int\limits_\Omega  {{f^2}d\Omega }  < \infty } \right\},\\ \\
 {H^1}(\Omega ) = \left\{ {f \in {L^2}(\Omega ),\,\,\,\nabla f \in {L^2}(\Omega )} \right\}, \\
  \\
 H_0^1(\Omega ) = \left\{ {f \in {H^1}(\Omega ),\,\,{{\left. f \right|}_{\partial \Omega }} = 0} \right\}, \\
 \\
 {H^k}(\Omega ) = \left\{ {f \in {L^2}(\Omega ),\,\,\,\,\,\,{D^\beta }f \in {L^2}(\Omega )\,\,\,for\,\,all\,\,\left| \beta  \right| \le k} \right\},
 \end{array}\]
and the derivative
\[{D^\alpha }f = \left( {\frac{{{\partial ^{{\alpha _1}}}f}}{{\partial x_1^{{\alpha _1}}}}} \right)\left( {\frac{{{\partial ^{{\alpha _2}}}f}}{{\partial x_2^{{\alpha _2}}}}} \right) \ldots \left( {\frac{{{\partial ^{{\alpha _p}}}f}}{{\partial x_p^{{\alpha _p}}}}} \right),\hspace{1cm}\left| \alpha
\right| = \sum\limits_{i = 1}^p {{\alpha _i}}.\]
The corresponding  inner products for $L^2(\Omega)$ and $H^1(\Omega)$ are as follows
\[\left( {f,g} \right)  = \int\limits_\Omega  {f(x)g(x)d\Omega } ,\,\,\,\,\,\,\,\,\,\,\,\,\,\,\,\,\,\,\,\,\,\,{\left( {f,g} \right) _1} = \left( {f,g} \right)  + \left( {\nabla f,\nabla g} \right) ,\]
and the associated norms   are
\[{\left\| f \right\|_{{L^2}(\Omega )}} = {\left( {f,f} \right) ^{\frac{1}{2}}},\,\,\,\,\,\,\,\,\,\,\,\,\,\,\,\,\,\,{\left\| f \right\|_{{H^1}(\Omega )}} = \left( {f,f} \right) _1^{\frac{1}{2}},\,\,\,\,\,\,\,\,\,\,\,\,\,\,\,\,\,\,\,{\left| f \right|_1} = {\left( {\nabla f,\nabla f} \right) ^{\frac{1}{2}}}.\]
Furthermore, associated norm for the space $H^m$ is as
\[{\left\| f \right\|_{{H^m}(\Omega )}} = {\left( {\sum\limits_{0 \le \left| \alpha  \right| \le m} {\left\| {{D^\alpha }f} \right\|_{{L^2}(\Omega )}^2} } \right)^{\frac{1}{2}}}.\]
To discretize the time variable, we define
  $${t_{n}} = n\tau,\qquad\forall~ n = 0,1, \ldots ,N,$$ where
$\tau= T/N$ is the step size. We introduce additionally
\[v_{}^{n - \frac{1}{2}} = v(x,y,{t_{n - \frac{1}{2}}}) = \frac{1}{2}\left( {v_{}^n + v_{}^{n - 1}} \right),\qquad{\delta _t}v_{}^{n - \frac{1}{2}} = \frac{1}{\tau }\left( {v_{}^n - v_{}^{n - 1}} \right),\qquad{v^n} = v(x,y,{t_n}).\]
The Crank-Nicolson scheme for problem \eqref{1gw} is as follows
\begin{equation}\label{1s}
\left\{ {\begin{array}{*{20}{l}}
{\displaystyle\frac{{\partial {u^{n - \frac{1}{2}}}}}{{\partial t}} + \xi ({ \mathbfit{x}})\nabla {u^{n - \frac{1}{2}}} - \nabla \cdot\left( {\zeta ({ \mathbfit{x}})\nabla {u^{n - \frac{1}{2}}}} \right) + {w_p}{e_1}f\left( {{u^{n - \frac{1}{2}}},{v^{n - \frac{1}{2}}}} \right) = \dot{W},}\\
{}\\
{\displaystyle\frac{{\partial {v^{n - \frac{1}{2}}}}}{{\partial t}} + \xi ({ \mathbfit{x}})\nabla {v^{n - \frac{1}{2}}} - \nabla \cdot\left( {\zeta({ \mathbfit{x}})\nabla {v^{n - \frac{1}{2}}}} \right) + {w_p}{e_2}f\left( {{u^{n - \frac{1}{2}}},{v^{n - \frac{1}{2}}}} \right) = \dot{W},}\\
{}\\
{\displaystyle\frac{{\partial {w^{n - \frac{1}{2}}}}}{{\partial t}} + \xi ({ \mathbfit{x}})\nabla {w^{n - \frac{1}{2}}} - \nabla \cdot\left( {\zeta({ \mathbfit{x}})\nabla {w^{n - \frac{1}{2}}}} \right) + {w_p}{e_3}f\left( {{u^{n - \frac{1}{2}}},{v^{n - \frac{1}{2}}}} \right) + r({ \mathbfit{x}}){w^{n - \frac{1}{2}}} = \dot{W},}
\end{array}} \right.
\end{equation}
where $C$ is a positive constant such that
$\left| {R_V^\tau } \right|{\rm and} \left| {R_B^\tau } \right| \le C{\tau ^2}.$
Discretizing relation \eqref{1s} yields
\begin{equation}\label{2s}\fontsize{10}{18}
\left\{ {\begin{array}{*{20}{l}}
{\displaystyle\frac{{{u^n} - {u^{n - 1}}}}{{\tau}} + \xi ({ \mathbfit{x}})\left[ {\frac{{\nabla {u^n} + \nabla {u^{n - 1}}}}{2}} \right] - \nabla \cdot\left[ {\zeta ({ \mathbfit{x}})\left( {\frac{{\nabla {u^n} + \nabla {u^{n - 1}}}}{2}} \right)} \right] + {w_p}{e_1}f\left( {{u^{n - \frac{1}{2}}},{v^{n - \frac{1}{2}}}} \right) = \frac{W^n-W^{n-1}}{\tau},}\\
{}\\ \\
{\displaystyle\frac{{{v^n} - {v^{n - 1}}}}{{\tau}} + \xi ({ \mathbfit{x}})\left[ {\frac{{\nabla {v^n} + \nabla {v^{n - 1}}}}{2}} \right] - \nabla \cdot\left[ {\zeta ({ \mathbfit{x}})\left( {\frac{{\nabla {v^n} + \nabla {v^{n - 1}}}}{2}} \right)} \right] + {w_p}{e_2}f\left( {{u^{n - \frac{1}{2}}},{v^{n - \frac{1}{2}}}} \right) = \frac{W^n-W^{n-1}}{\tau},}\\
{}\\ \\
\begin{array}{l}
\displaystyle\frac{{{w^n} - {w^{n - 1}}}}{{\tau}} + \xi ({ \mathbfit{x}})\left[ {\frac{{\nabla {w^n} + \nabla {w^{n - 1}}}}{2}} \right] - \nabla \cdot\left[ {\zeta ({ \mathbfit{x}})\left( {\frac{{\nabla {w^n} + \nabla {w^{n - 1}}}}{2}} \right)} \right]\vspace{0.2cm}
\\ 
\hspace{6.2cm}+ {w_p}{e_3}f\left( {{u^{n - \frac{1}{2}}},{v^{n - \frac{1}{2}}}} \right) + r({ \mathbfit{x}})\left[ {\displaystyle\frac{{{w^n} + {w^{n - 1}}}}{2}} \right] = {\displaystyle\frac{W^n-W^{n-1}}{\tau}},\\
\end{array}
\end{array}} \right.
\end{equation}
or
\begin{equation}\label{3s}
\left\{ {\begin{array}{*{20}{l}}
\begin{array}{l}
{u^n} +\displaystyle \frac{{\tau}}{2}\xi ({ \mathbfit{x}})\nabla {u^n} - \frac{{\tau}}{2}\nabla \cdot\left[ {\zeta ({ \mathbfit{x}})\nabla {u^n}} \right]+W^n\\
\\
\,\,\,\,\,\,\,\,\,\,\,\,\,\,\,\,\,\,\,\,\,\,\,\, = {u^{n - 1}} - \displaystyle\frac{{\tau}}{2}\xi ({ \mathbfit{x}})\nabla {u^{n - 1}} + \frac{{\tau}}{2}\nabla \cdot\left[ {\zeta ({ \mathbfit{x}})\nabla {u^{n - 1}}} \right] - \tau{w_p}{e_1}f\left( {{u^{n - 1}},{v^{n - 1}}} \right)+W^{n-1},\\
\\
{v^n} + \displaystyle\frac{{dt}}{2}\xi ({ \mathbfit{x}})\nabla {v^n} - \frac{{\tau}}{2}\nabla \cdot\left[ {\zeta ({ \mathbfit{x}})\nabla {v^n}} \right]+W^n\\
\\
\,\,\,\,\,\,\,\,\,\,\,\,\,\,\,\,\,\,\,\,\,\,\,\, = {v^{n - 1}} -\displaystyle \frac{{dt}}{2}\xi ({ \mathbfit{x}})\nabla {v^{n - 1}} + \displaystyle\frac{{dt}}{2}\nabla \cdot\left[ {\zeta ({ \mathbfit{x}})\nabla {v^{n - 1}}} \right] - dt{w_p}{e_2}f\left( {{u^{n - 1}},{v^{n - 1}}} \right)+W^{n-1},
\end{array}\\
{}\\
\begin{array}{l}
\left( {1 +\displaystyle \frac{{dt}}{2}r({ \mathbfit{x}})} \right){w^n} + \frac{{dt}}{2}\xi ({ \mathbfit{x}})\nabla {w^n} -\displaystyle \frac{{dt}}{2}\nabla \cdot\left[ {\zeta ({ \mathbfit{x}})\nabla {w^n}} \right]+W^n\\
\\
 = \left( {1 -\displaystyle \frac{{dt}}{2}r({ \mathbfit{x}})} \right){w^{n - 1}} -\displaystyle \frac{{dt}}{2}\xi ({ \mathbfit{x}})\nabla {w^{n - 1}} + \frac{{dt}}{2}\nabla \cdot\left[ {\zeta ({ \mathbfit{x}})\nabla {w^{n - 1}}} \right] - dt{w_p}{e_3}f\left( {{u^{n - 1}},{v^{n - 1}}} \right)+W^{n-1}
\end{array}
\end{array}} \right.
\end{equation}
The vector-matrix configuration of Eq.  \eqref{3s} is
\begin{equation}\label{6sss}
 \mathbfit{H}_1{{\bm{{\cal U}}}^n} + \frac{{\tau}}{2}{\mathbfit{I}}\nabla {{\bm{{\cal U}}}^n} - \frac{{\tau}}{2}{\mathbfit{I}}\nabla  \cdot \zeta ({\mathbfit{x}})\nabla {{\bm{{\cal U}}}^n}+W^{n} = {\mathbfit{H}}_2{{ \bm{{\cal U}}}^{n - 1}} - \frac{{\tau}}{2}{\mathbfit{I}}\nabla {{\bm{{\cal U}}}^{n - 1}} + \frac{{\tau}}{2}{\mathbfit{I}}\nabla  \cdot \zeta ({\mathbfit{x}})\nabla {{\bm{{\cal U}}}^{n - 1}} - \tau{\mathbfit{NF}}\left( {{{\bm{{\cal U}}}^{n - 1}}} \right)+W^{n-1},
\end{equation}
where $\bf{I}$ is the identity matrix and
\begin{equation}\label{7s}
{\mathbfit{H}_1} = \rm{diag}\left( {1,1,1 + \frac{\tau }{2}r(\mathbfit{x})} \right),\,\,\,\,\,{\mathbfit{H}_2} = \rm{diag}\left( {1,1,1 - \frac{\tau }{2}r(\mathbfit{x})} \right),\,\,\,\,{\mathbfit{N}} = \rm{diag}\left( {{w_p}{e_1},{w_p}{e_2},{w_p}{e_3}} \right),\,\,\,\,\,
\end{equation}
and also the unknown vector is ${\bm{{\cal U}}}=(u,v,w)$.
\subsection{Error analysis of the semi-discrete formulation}
\begin{Theorem}\label{Th1}
If $\bm{{\cal U}}^n\in { \mathbfit H}_0^1(\Omega)$, then relation \eqref{6sss} will be
unconditionally stable.
\end{Theorem}
\begin{proof}
Let $\zeta(x)$ and $\xi(x)\in L^2(\Omega)$.
We want to  find $\bm{{\cal U}}^n\in { \mathbfit H}_0^1(\Omega)$ such that 
{\allowdisplaybreaks
\begin{eqnarray}\label{7s}
{{ \mathbfit{H}}_1}\left( {{{ \bm{{\cal U}}}^n},{ \bm{\chi}}} \right)  &+& \frac{{\tau}}{2}{ \mathbfit{I}}\left( {\zeta (x)\nabla {{ \bm{{\cal U}}}^n},\nabla { \bm{\chi}}} \right)  - \frac{{\tau}}{2}{ \mathbfit{I}}\left( {{{ \bm{{\cal U}}}^n},\frac{\partial }{{\partial x}}{ \bm{\chi}}} \right)  - \frac{{\tau}}{2}{ \mathbfit{I}}\left( {{{ \bm{{\cal U}}}^n},\frac{\partial }{{\partial y}}{ \bm{\chi}}} \right) +  \left( W^n ,\textbf{V}
\right)
 \nonumber\\\nonumber
\\
 &=& {{ \mathbfit{H}}_2}\left( {{{ \bm{{\cal U}}}^{n - 1}},{ \bm{\chi}}} \right)  - \frac{{\tau}}{2}{ \mathbfit{I}}\left( {\zeta (x)\nabla {{ \bm{{\cal U}}}^{n - 1}},\nabla { \bm{\chi}}} \right)  + \frac{{\tau}}{2}{ \mathbfit{I}}\left( {{{ \bm{{\cal U}}}^{n - 1}},\frac{\partial }{{\partial x}}{ \bm{\chi}}} \right) \nonumber\\\nonumber
\\
 &+& \frac{{\tau}}{2}{ \mathbfit{I}}\left( {{{ \bm{{\cal U}}}^{n - 1}},\frac{\partial }{{\partial y}}{ \bm{\chi}}} \right)  - \tau{ \mathbfit{N}}\left( {{ \mathbfit{F}},{ \bm{\chi}}} \right) +\left( W^{n-1} ,\textbf{V}
 \right)  \qquad \forall~{ \bm{\chi}} \in { \mathbfit{H}}_0^1(\Omega ).
\end{eqnarray}
\allowdisplaybreaks}
Let ${{  \bm{\widetilde {{\cal U}}}}^n}$ be an approximate solution of ${{  \mathbfit{  {U}}}^n}$, then
{\allowdisplaybreaks
\begin{eqnarray}\label{7s-11}
{{ \mathbfit{H}}_1}\left( {{{  \bm{\widetilde {{\cal U}}}}^n},{ \bm{\chi}}} \right)  &+& \frac{{\tau}}{2}{ \mathbfit{I}}\left( {\zeta (x)\nabla {{  \bm{\widetilde {{\cal U}}}}^n},\nabla { \bm{\chi}}} \right)  - \frac{{\tau}}{2}{ \mathbfit{I}}\left( {{{  \bm{\widetilde {{\cal U}}}}^n},\frac{\partial }{{\partial x}}{ \bm{\chi}}} \right)  - \frac{{\tau}}{2}{ \mathbfit{I}}\left( {{{ \bm{{\cal U}}}^n},\frac{\partial }{{\partial y}}{ \bm{\chi}}} \right) +\left( W^{n} ,\textbf{V}
\right) \nonumber\\\nonumber
\\
 &=& {{ \mathbfit{H}}_2}\left( {{{  \bm{\widetilde {{\cal U}}}}^{n - 1}},{ \bm{\chi}}} \right)  - \frac{{\tau}}{2}{ \mathbfit{I}}\left( {\zeta (x)\nabla {{  \bm{\widetilde {{\cal U}}}}^{n - 1}},\nabla { \bm{\chi}}} \right)  + \frac{{\tau}}{2}{ \mathbfit{I}}\left( {{{  \bm{\widetilde {{\cal U}}}}^{n - 1}},\frac{\partial }{{\partial x}}{ \bm{\chi}}} \right) \nonumber\\\nonumber
\\
 &+& \frac{{\tau}}{2}{ \mathbfit{I}}\left( {{{  \bm{\widetilde {{\cal U}}}}^{n - 1}},\frac{\partial }{{\partial y}}{ \bm{\chi}}} \right)  - \tau{ \mathbfit{N}}\left( {{ \mathbfit{\widetilde {F}}},{ \bm{\chi}}} \right) +\left( W^{n-1} ,\textbf{V}
 \right) \qquad\forall ~{ \bm{\chi}} \in { \mathbfit{H}}_0^1(\Omega ),
\end{eqnarray}
\allowdisplaybreaks}
where $\widetilde {\mathbfit{F}} = {\mathbfit{F}}( {\widetilde {\bm{{\cal U}}}} )$.  Subtracting Eq. \eqref{7s-11} for Eq. \eqref{7s} , results
{\allowdisplaybreaks
\begin{eqnarray}\label{7s-12}
{{ \mathbfit{H}}_1}\left( {{{ \mathbfit{\Psi}}^n},{ \bm{\chi}}} \right)  &+& \frac{{\tau}}{2}{ \mathbfit{I}}\left( {\zeta (x)\nabla {{ \mathbfit{\Psi}}^n},\nabla { \bm{\chi}}} \right)  - \frac{{\tau}}{2}{ \mathbfit{I}}\left( {{{ \mathbfit{\Psi}}^n},\frac{\partial }{{\partial x}}{ \bm{\chi}}} \right)  - \frac{{\tau}}{2}{ \mathbfit{I}}\left( {{{\mathbfit{\Psi}}^n},\frac{\partial }{{\partial y}}{ \bm{\chi}}} \right) \nonumber\\\nonumber
\\
 &=& {{ \mathbfit{H}}_2}\left( {{{ \mathbfit{\Psi}}^{n - 1}},{ \bm{\chi}}} \right)  - \frac{{\tau}}{2}{ \mathbfit{I}}\left( {\zeta (x)\nabla {{ \mathbfit{\Psi}}^{n - 1}},\nabla { \bm{\chi}}} \right)  + \frac{{\tau}}{2}{ \mathbfit{I}}\left( {{{ \mathbfit{\Psi}}^{n - 1}},\frac{\partial }{{\partial x}}{ \bm{\chi}}} \right) \nonumber\\\nonumber
\\
 &+& \frac{{\tau}}{2}{ \mathbfit{I}}\left( {{{ \mathbfit{\Psi}}^{n - 1}},\frac{\partial }{{\partial y}}{ \bm{\chi}}} \right)  - \tau{ \mathbfit{N}}\left( {{ \mathbfit{F-\widetilde {F}}},{ \bm{\chi}}} \right) ,{\mkern 1mu} {\mkern 1mu} \qquad\forall~{ \bm{\chi}} \in { \mathbfit{H}}_0^1(\Omega ),
\end{eqnarray}
\allowdisplaybreaks}
where
\[\mathbfit{\Psi}^n = \mathbb{E} [{{\bm{{\cal U}}}^n} - {\widetilde {\bm{{\cal U}}}^n}].\]
Setting ${ \bm{\chi}}=\mathbfit{\Psi} ^n$ in Eq. \eqref{7s-12} yields
{\allowdisplaybreaks
\begin{eqnarray}\label{8s}
{{ \mathbfit{H}}_1}\left( {{{ \mathbfit{\Psi}}^n},{ \mathbfit{\Psi}^n}} \right)  &+& \frac{{\tau}}{2}{ \mathbfit{I}}\left( {\zeta (x)\nabla {{ \mathbfit{\Psi}}^n},\nabla { \mathbfit{\Psi}^n}} \right)  - \frac{{\tau}}{2}{ \mathbfit{I}}\left( {{{ \mathbfit{\Psi}}^n},\frac{\partial }{{\partial x}}{ \mathbfit{\Psi}^n}} \right)  - \frac{{\tau}}{2}{ \mathbfit{I}}\left( {{{\mathbfit{\Psi}}^n},\frac{\partial }{{\partial y}}{ \mathbfit{\Psi}^n}} \right) \nonumber\\\nonumber
\\
 &=& {{ \mathbfit{H}}_2}\left( {{{ \mathbfit{\Psi}}^{n - 1}},{ \mathbfit{\Psi}^n}} \right)  - \frac{{\tau}}{2}{ \mathbfit{I}}\left( {\zeta (x)\nabla {{ \mathbfit{\Psi}}^{n - 1}},\nabla { \mathbfit{\Psi}^n}} \right)  + \frac{{\tau}}{2}{ \mathbfit{I}}\left( {{{ \mathbfit{\Psi}}^{n - 1}},\frac{\partial }{{\partial x}}{ \mathbfit{\Psi}^n}} \right) \nonumber\\\nonumber
\\
 &+& \frac{{\tau}}{2}{ \mathbfit{I}}\left( {{{ \mathbfit{\Psi}}^{n - 1}},\frac{\partial }{{\partial y}}{ \bm{\chi}}} \right)  - \tau{ \mathbfit{N}}\left( {{ \mathbfit{F-\widetilde {F}}},{ \mathbfit{\Psi}^n}} \right).
\end{eqnarray}
\allowdisplaybreaks}
Applying the  Cauchy-Schwarz inequality for Eq. \eqref{8s}, results
{\allowdisplaybreaks
\begin{eqnarray*}
\left\| {{{ \mathbfit{H}}_1}} \right\|\left\| {{{ \mathbfit{\Psi}}^n}} \right\|_{{{ \bm{L}}^2}(\Omega )}^2 &+& \frac{{\tau}}{2}\left\| {\zeta (x)} \right\|\left\| {\nabla {{ \mathbfit{\Psi}}^n}} \right\|_{{{ \bm{L}}^2}(\Omega )}^2 \le \frac{{\tau}}{2}\left( {{{ \mathbfit{\Psi}}^n},\frac{\partial }{{\partial x}}{{ \mathbfit{\Psi}}^n}} \right)  + \frac{{\tau}}{2}\left( {{{ \mathbfit{\Psi}}^n},\frac{\partial }{{\partial y}}{{ \mathbfit{\Psi}}^n}} \right) \\
\\
 &+& \left\| {{{ \mathbfit{H}}_2}} \right\|{\left\| {{{ \mathbfit{\Psi}}^n}} \right\|_{{{ \bm{L}}^2}(\Omega )}}{\left\| {{{ \mathbfit{\Psi}}^{n - 1}}} \right\|_{{{ \bm{L}}^2}(\Omega )}} + \frac{{\tau}}{2}\left\| {\zeta (x)} \right\|{\left\| {\nabla {{ \mathbfit{\Psi}}^n}} \right\|_{{{ \bm{L}}^2}(\Omega )}}{\left\| {\nabla {{ \mathbfit{\Psi}}^{n - 1}}} \right\|_{{{ \bm{L}}^2}(\Omega )}}  \\
\\
 &+& \frac{{\tau}}{2}\left( {{{ \mathbfit{\Psi}}^{n - 1}},\frac{\partial }{{\partial y}}{{ \mathbfit{\Psi}}^n}} \right) + \frac{{\tau}}{2}\left( {{{ \mathbfit{\Psi}}^{n - 1}},\frac{\partial }{{\partial x}}{{ \mathbfit{\Psi}}^n}} \right) - \tau{ \mathbfit{N}}\left( {{ \mathbfit{F-\widetilde {F}}},{ \mathbfit{\Psi}^n}} \right).
\end{eqnarray*}
\allowdisplaybreaks}
There exists constant $C$ such that
\begin{equation}
  \left\| {{{ \mathbfit{H}}_2}} \right\|,\left\| {{{ \mathbfit{H}}_3}} \right\| \le  C,
\end{equation}
and
\begin{equation}
\left\| {{ \mathbfit{F}} - \widetilde { \mathbfit{F}}} \right\| \le  \bm{L}{{\mathbfit{\Psi}}^{n-1}}.
\end{equation}
By simplification we have
{\allowdisplaybreaks
\begin{eqnarray*}
\left\| {{{ \mathbfit{H}}_1}} \right\|\left\| {{{ \mathbfit{\Psi}}^n}} \right\|_{{{ \bm{L}}^2}(\Omega )}^2 &+& \frac{{\tau}}{2}\left\| {\zeta (x)} \right\|\left\| {\nabla {{ \mathbfit{\Psi}}^n}} \right\|_{{{ \bm{L}}^2}(\Omega )}^2 \le \frac{{\tau}}{2}{\left\| {{{ \mathbfit{\Psi}}^n}} \right\|_{{{ \bm{L}}^2}(\Omega )}}{\left\| {\nabla {{ \mathbfit{\Psi}}^{n - 1}}} \right\|_{{{ \bm{L}}^2}(\Omega )}} + \frac{{\tau}}{2}{\left\| {{{ \mathbfit{\Psi}}^{n - 1}}} \right\|_{{{ \bm{L}}^2}(\Omega )}}{\left\| {\nabla {{ \mathbfit{\Psi}}^n}} \right\|_{{{ \bm{L}}^2}(\Omega )}}\\
\\
 &+& \left\| {{{ \mathbfit{H}}_2}} \right\|{\left\| {{{ \mathbfit{\Psi}}^n}} \right\|_{{{ \bm{L}}^2}(\Omega )}}{\left\| {{{ \mathbfit{\Psi}}^{n - 1}}} \right\|_{{{ \bm{L}}^2}(\Omega )}} + \frac{{\tau}}{2}\left\| {\zeta (x)} \right\|{\left\| {\nabla {{ \mathbfit{\Psi}}^n}} \right\|_{{{ \bm{L}}^2}(\Omega )}}{\left\| {\nabla {{ \mathbfit{\Psi}}^{n - 1}}} \right\|_{{{ \bm{L}}^2}(\Omega )}}\\
\\
 &+& \tau \bm{L}\left\| { \mathbfit{N}} \right\|{\left\| {{{ \mathbfit{\Psi}}^{n - 1}}} \right\|_{{{ \bm{L}}^2}(\Omega )}}{\left\| {{{ \mathbfit{\Psi}}^{n}}} \right\|_{{{ \bm{L}}^2}(\Omega )}}.
\end{eqnarray*}
\allowdisplaybreaks}
So, from the following assumption and the definition of matrices ${{{ \mathbfit{H}}_1}}$ and ${{{ \mathbfit{H}}_2}}$, we have
\[\left\| {{{ \mathbfit{H}}_2}} \right\| \le \left\| {{{ \mathbfit{H}}_1}} \right\|.\]
Now, we can get
{\allowdisplaybreaks
\begin{eqnarray}\label{9s}
\frac{1}{2}\left\| {{{ \mathbfit{H}}_1}} \right\|\left\| {{{ \mathbfit{\Psi}}^n}} \right\|_{{{ \bm{L}}^2}(\Omega )}^2 &+& \frac{{\tau}}{4}\left\| {\zeta ({ \mathbfit{x}})} \right\|\left\| {\nabla {{ \mathbfit{\Psi}}^n}} \right\|_{{{ \bm{L}}^2}(\Omega )}^2 \\ \nonumber
\\
&\le& \frac{1}{2}\left\| {{{ \mathbfit{H}}_1}} \right\|\left\| {{{ \mathbfit{\Psi}}^{n - 1}}} \right\|_{{{ \bm{L}}^2}(\Omega )}^2 + \frac{{\tau}}{2}\left\| {\zeta ({ \mathbfit{x}})} \right\|\left\| {\nabla {{ \mathbfit{\Psi}}^{n - 1}}} \right\|_{{{ \bm{L}}^2}(\Omega )}^2\nonumber\\\nonumber
\\\nonumber
&+&\frac{{C_1 \bm{L}\tau}}{{2\left\| {\zeta ({ \mathbfit{x}})} \right\|}}\left\| {{{ \mathbfit{\Psi}}^n}} \right\|_{{{ \bm{L}}^2}(\Omega )}^2 + \frac{{C_2 \bm{L}\tau}}{{2\left\| {\zeta ({ \mathbfit{x}})} \right\|}}\left\| {{{ \mathbfit{\Psi}}^{n - 1}}} \right\|_{{{ \bm{L}}^2}(\Omega )}^2.
\end{eqnarray}
\allowdisplaybreaks}
Using the below relation
\[\left\| {{{ \mathbfit{\Psi}}^n}} \right\|_{{{ \mathbfit{H}}_w}(\Omega )}^2 = \left\| {{{ \mathbfit{H}}_1}} \right\|\left\| {{{ \mathbfit{\Psi}}^n}} \right\|_{{{ \bm{L}}^2}(\Omega )}^2 + \frac{1}{2}\tau\left\| {\zeta ({ \mathbfit{x}})} \right\|\left\| {\nabla {{ \mathbfit{\Psi}}^n}} \right\|_{{{ \bm{L}}^2}(\Omega )}^2,\]
Eq. \eqref{9s} is changed to
\begin{equation}\label{11s}
\left\| {{{ \mathbfit{\Psi}}^n}} \right\|_{{{ \mathbfit{H}}_w}(\Omega )}^2 \le \left\| {{{ \mathbfit{\Psi}}^{n - 1}}} \right\|_{{{ \mathbfit{H}}_w}(\Omega )}^2 + \frac{{C_1 \bm{L}\tau}}{{\left\| {\zeta ({ \mathbfit{x}})} \right\|}}\left\| {{{ \mathbfit{\Psi}}^n}} \right\|_{{{ \mathbfit{H}}_w}(\Omega )}^2 + \frac{{C_2 \bm{L}\tau}}{{\left\| {\zeta ({ \mathbfit{x}})} \right\|}}\left\| {{{ \mathbfit{\Psi}}^{n - 1}}} \right\|_{{{ \mathbfit{H}}_w}(\Omega )}^2. 
\end{equation}
By summing Eq. \eqref{11s} for $j$ from 0 to $n$, gives
\[\begin{array}{l}
\displaystyle\sum\limits_{m = 1}^n {\left\| {{{ \mathbfit{\Psi}}^m}} \right\|_{{{ \mathbfit{H}}_w}(\Omega )}^2}  \le \sum\limits_{m = 1}^n {\left\| {{{ \mathbfit{\Psi}}^{m - 1}}} \right\|_{{{ \mathbfit{H}}_w}(\Omega )}^2}  + \frac{{C_1 \bm{L}\tau}}{{\left\| {\zeta ({ \mathbfit{x}})} \right\|}}\sum\limits_{m = 1}^n {\left\| {{{ \mathbfit{\Psi}}^m}} \right\|_{{{ \mathbfit{H}}_w}(\Omega )}^2}  +\displaystyle \frac{{C_2 \bm{L}\tau}}{{\left\| {\zeta ({ \mathbfit{x}})} \right\|}}\sum\limits_{m = 1}^n {\left\| {{{ \mathbfit{\Psi}}^{m - 1}}} \right\|_{{{ \mathbfit{H}}_w}(\Omega )}^2}.
\end{array}\]
Thus, we have
\begin{equation}\label{12s}
\left\| {{{ \mathbfit{\Psi}}^n}} \right\|_{{{ \mathbfit{H}}_w}(\Omega )}^2 \le \left\| {{{ \mathbfit{\Psi}}^0}} \right\|_{{{ \mathbfit{H}}_w}(\Omega )}^2 + \frac{{2C \bm{L}\tau}}{{\left\| {\zeta ({ \mathbfit{x}})} \right\|}}\sum\limits_{m = 1}^n {\left\| {{{ \mathbfit{\Psi}}^m}} \right\|_{{{ \mathbfit{H}}_w}(\Omega )}^2}  
\end{equation}
Considering  Gronwall's inequality for Eq. (\ref{12s}) yields
{\allowdisplaybreaks
\begin{eqnarray*}
\left\| {{{ \mathbfit{\Psi}}^n}} \right\|_{{{ \mathbfit{H}}_w}(\Omega )}^2 &\le& \left\| {{{ \mathbfit{\Psi}}^0}} \right\|_{{{ \mathbfit{H}}_w}(\Omega )}^2 + \frac{{2C \bm{L}\tau}}{{\left\| {\zeta ({ \mathbfit{x}})} \right\|}}\sum\limits_{m = 1}^n {\left\| {{{ \mathbfit{\Psi}}^m}} \right\|_{{{ \mathbfit{H}}_w}(\Omega )}^2} \\
\\
 &\le& \left\{ {\left\| {{{ \mathbfit{\Psi}}^0}} \right\|_{{{ \mathbfit{H}}_w}(\Omega )}^2 } \right\}\exp \left( {\frac{{2C \bm{L}n\tau}}{{\left\| {\zeta ({ \mathbfit{x}})} \right\|}}} \right)\le \mathbf{C}\left\| {{{ \mathbfit{\Psi}}^0}} \right\|_{{{ \mathbfit{H}}_w}(\Omega )}^2.
\end{eqnarray*}
\allowdisplaybreaks}
So, we have
\[\left\| {{{ \mathbfit{\Psi}}^n}} \right\|_{{{ \bm{L}}^2}(\Omega )}^{} \le \left\| {{{ \mathbfit{\Psi}}^n}} \right\|_{{{ \mathbfit{H}}_w}(\Omega )}^{} \le C\left\| {{{ \mathbfit{\Psi}}^0}} \right\|_{{{ \mathbfit{H}}_w}(\Omega )}^{}.\]
\end{proof} 
\begin{Theorem} The convergence order of relation (\ref{6sss}) is 
 ${\cal O}\left( {{\tau ^2}}\right)$.
\end{Theorem}
\begin{proof}
Let us assume $\textbf{u}^n,\,{{\bm{{\cal U}}}^n}\in { \mathbfit H}_0^1(\Omega)$. We set
\[{ \mathbfit{\rm X} ^n} = \mathbb{E} [ \textbf{u}^n - {{ \bm{{\cal U}}}^n}] \qquad n \ge 1,\] 
where ${ \mathbfit{\rm X} ^0}={ \mathbfit 0}$.
Then, we have
\begin{equation}\label{6s}
\begin{array}{l}
 \mathbfit{H}_1{ \mathbfit{\rm X}}^n + \displaystyle\frac{{\tau}}{2}{\mathbfit{I}}\nabla {{ \mathbfit{\rm X}}^n} - \frac{{\tau}}{2}{\mathbfit{I}}\nabla  \cdot \zeta ({\mathbfit{x}})\nabla {{ \mathbfit{\rm X}}^n}= \\ \\ {\mathbfit{H}}_2{{ \mathbfit{\rm X}}^{n - 1}} - \displaystyle\frac{{\tau}}{2}{\mathbfit{I}}\nabla {{ \mathbfit{\rm X}}^{n - 1}} + \frac{{\tau}}{2}{\mathbfit{I}}\nabla  \cdot \zeta ({\mathbfit{x}})\nabla {{ \mathbfit{\rm X}}^{n - 1}} +\tau \mathbfit{ R} - \tau{\mathbfit{N}}\left({{ \mathbfit{F}}^{n-1} - \widetilde { \mathbfit{F}}^{n-1}} \right).
\end{array}
\end{equation}
According to the Crank-Nicolson idea, we have
\[\left| {\bf{R}} \right| \le {C_1}{\tau ^2}.\]
Similar to Theorem \ref{Th1}, we obtain
{\allowdisplaybreaks
\begin{eqnarray*}
\left\| {{{\mathbfit{\rm X}}^n}} \right\|_{{{ \mathbfit{H}}_w}(\Omega )}^2 &\le& \left\| {{{\mathbfit{\rm X}}^0}} \right\|_{{{ \mathbfit{H}}_w}(\Omega )}^2 + \frac{{2\bm{L}\tau}}{{\left\| {\zeta ({ \mathbfit{x}})} \right\|}}\sum\limits_{m = 1}^n {\left\| {{{\mathbfit{\rm X}}^m}} \right\|_{{{ \mathbfit{H}}_w}(\Omega )}^2}  + \mathop {\max }\limits_{1 \le m \le n} \left\| {{{ \mathbfit{R}}{}}} \right\|_{{{ \bm{L}}^2}(\Omega )}^2\\
\\
 &\le& \left\{ {  \mathop {\max }\limits_{1 \le m \le n} \left\| {{{ \mathbfit{R}}}} \right\|_{{{ \bm{L}}^2}(\Omega )}^2} \right\}\exp \left( {\frac{{2\bm{L}n\tau}}{{\left\| {\zeta ({ \mathbfit{x}})} \right\|}}} \right)\le \mathbf{C}\tau^2\\ 
 \\
 &\le& \left\{ {\mathop {\max }\limits_{1 \le m \le n} \left\| \bm{R} \right\|_{{\bm{L}^2}(\Omega )}^2} \right\}\exp \left( {\frac{{2\bm{L}n\tau }}{{\left\| {\zeta (\bm{x})} \right\|}}} \right)\\
 \\
 &\le& \exp \left( {\frac{{2\bm{L}n\tau }}{{\left\| {\zeta (\bm{x})} \right\|}}} \right){C_1}{\tau ^2} \le {\bf{C}}{\tau ^2}.
\end{eqnarray*}
\allowdisplaybreaks}
which completes the proof.
\end{proof}
\section{Error estimation for full-discrete plane}
\label{section3}
In this section, we employ a new class of Legendre polynomial functions  which were developed in \cite{Shen}.
\begin{lemma}\cite{Shen}
	Consider the following relations
	\begin{equation}\label{basis}
	{\psi _k}(x) = {\gamma _k}({L_k}(x) - {L_{k + 2}}(x)),
	\end{equation}
	in which ${\gamma _k} = {\left( {4k + 6} \right)^{ - \frac{1}{2}}}$ and $L_k(x)$ are the Legendre polynomials. Let us denote
	\begin{equation}
	{a_{jk}} = \int\limits_{ - 1}^1 {\frac{{d{\psi _k}(x)}}{{dx}}\frac{{d{\psi _j}(x)}}{{dx}}} dx,\,\,\,\,\,\,\,\,\,\,\,\,\,\,\,\,\,\,\,\,\,\,\,\,\,\,\,{b_{jk}} = \int\limits_{ - 1}^1 {{\psi _k}(x){\psi _j}(x)} dx.
	\end{equation}
	Then
	\begin{equation}
	{a_{jk}} = \left\{ \begin{array}{l}
	1,\,\,\,\,\,\,\,\,\,\,\,\,\,\,\,\,\,k = j,\\ \\
	0,\,\,\,\,\,\,\,\,\,\,\,\,\,\,\,\,\,k \ne j,
	\end{array} \right.\,\,\,\,\,\,\,\,\,\,\,\,\,\,\,\,\,\,\,\,\,\,\,\,{b_{jk}} = {b_{kj}} = \left\{ \begin{array}{l}
	{\gamma _k}{\gamma _j}\left( {\displaystyle\frac{2}{{2j + 1}} + \frac{2}{{2j + 5}}} \right),\,\,\,\,\,\,\,\,\,\,\,\,\,\,\,\,\,\,\,\,k = j,\\
	\\
	- {\gamma _k}{\gamma _j}\displaystyle\frac{2}{{2k + 1}},\,\,\,\,\,\,\,\,\,\,\,\,\,\,\,\,\,\,\,\,k = j + 2,\\
	\\
	0,\,\,\,\,\,\,\,\,\,\,\,\,\,\,\,\,\,\,\,\,{\rm Otherwise}.
	\end{array} \right.
	\end{equation}
\end{lemma}
\hspace{-0.5cm}The SEM as a combination of the finite element method and spectral polynomials has been developed by Patera \cite{Patera}. By dividing the  computational region into $N_e$
non-overlapping elements $\Omega_e$
\[\Omega  = \bigcup\limits_{e = 1}^{{N_e}} {{\Omega _e}} ,\,\,\,\,\,\,\,\,\,\,\,\,\,\,{\Omega _i} \cap {\Omega _j} = \emptyset ,\,\,\,\,\,\,\,\,\,\,i \ne j.\]
Now, we define the following projection operator.
\begin{equation}\label{version4}
{\cal P}_h^1:H_0^1\left( \Omega  \right) \to {\bf{V}}_N
^0,
\end{equation}
where
\begin{equation}\label{version5}
\left( {\nabla \left( {u -{\cal P}_h^1u} \right),\nabla v} \right) =
0,\,\,\,\,\,\,\,\,\,\,\,u \in H_0^1\left( \Omega
\right),\,\,\,\,\,\,\,\,\,\,\,\,\forall v \in
{\bf{V}}_N ^0,
\end{equation}
and  ${\bf{V}}_N ^0$ is the spectral element
approximation space
\begin{equation}\label{version5-}
{\bf{V}}_N ^0 = \left\{ {w \in H_0^1\left( \Omega
	\right)\,\,:\,\,\,{{\left. w \right|}_{{\Omega _s}}} \in
	{\mathbb{P}_N}\left( \Omega  \right),\,\,\,s = 1,2, \ldots ,{n_s}}
\right\}.
\end{equation}
\begin{lemma}\cite{canuto2006spectral} Let $u\in H^\upsilon$ ($\upsilon \geq 1$), therefore
	\begin{equation}\label{version6-1}
	{\left\| {u - {\cal P}_h^1u} \right\|} \le C{\left[
		{\displaystyle\sum\limits_{k = 1}^{{n_s}} {h_k^{2\left( {\min \left(
						{{N_k} + 1,\upsilon } \right) - 1} \right)}N_k^{2\left( {1 - \upsilon }
					\right)}\left\| u \right\|_\upsilon ^2} } \right]^{\frac{1}{2}}}.
	\end{equation}
	In the special cases $N_k=N$ and  $h\leq h_k\leq c'h$ we get
	\begin{equation}\label{version6}
	\left\| {u - {\cal P} _h^1u} \right\| \le Ch_k^{\left( {\min \left( {N +
				1,\upsilon } \right) - 1} \right)}N_{}^{1 - \upsilon }\left\| u
	\right\|_\upsilon ^{}.
	\end{equation}
\end{lemma}
\hspace{-0.7cm}We aim to find a $ \bm{{\cal U}}^n\in \omega_r^d$ such that 
{\allowdisplaybreaks
{\small
 \begin{eqnarray}\label{25P}\footnotesize
{{ \mathbfit{H}}_1}\left( {{{ \bm{{\cal U}}}^n},{ \bm{\chi}}} \right)  &+& \frac{{\tau}}{2}{ \mathbfit{I}}\left( {\zeta (x)\nabla {{ \bm{{\cal U}}}^n},\nabla { \bm{\chi}}} \right)  - \frac{{\tau}}{2}{ \mathbfit{I}}\left( {{{ \bm{{\cal U}}}^n},\frac{\partial }{{\partial x}}{ \bm{\chi}}} \right)  - \frac{{\tau}}{2}{ \mathbfit{I}}\left( {{{ \bm{{\cal U}}}^n},\frac{\partial }{{\partial y}}{ \bm{\chi}}} \right)  +  \left( W^n ,\textbf{V}
\right)
\nonumber\\\nonumber
\\
 &=& {{ \mathbfit{H}}_2}\left( {{{ \bm{{\cal U}}}^{n - 1}},{ \bm{\chi}}} \right)  - \frac{{\tau}}{2}{ \mathbfit{I}}\left( {\zeta (x)\nabla {{ \bm{{\cal U}}}^{n - 1}},\nabla { \bm{\chi}}} \right)  + \frac{{\tau}}{2}{ \mathbfit{I}}\left( {{{ \bm{{\cal U}}}^{n - 1}},\frac{\partial }{{\partial x}}{ \bm{\chi}}} \right) \nonumber\\\nonumber
\\
 &+& \frac{{\tau}}{2}{ \mathbfit{I}}\left( {{{ \bm{{\cal U}}}^{n - 1}},\frac{\partial }{{\partial y}}{ \bm{\chi}}} \right)  - \tau{ \mathbfit{N}}\left( {{ \mathbfit{F}^{n-1}},{ \bm{\chi}}} \right)+ \tau\left( {{ \mathbfit{R}_{t}^n},{ \bm{\chi}}} \right)  +  \left( W^{n-1} ,\textbf{V}
 \right)
 \qquad{ \bm{\chi}} \in { \mathbfit{H}}_0^1(\Omega ).
\end{eqnarray}
}
\allowdisplaybreaks}
\hspace{-0.4cm}The spectral element formulation  is:
  find a $ \bm{{\cal U}}_h^n\in \omega_r^d$ such that 
{\allowdisplaybreaks
\begin{eqnarray}\label{26P}
{{ \mathbfit{H}}_1}\left( {{{ \bm{{\cal U}}}_h^n},{ \bm{\chi}_h}} \right)  &+& \frac{{\tau}}{2}{ \mathbfit{I}}\left( {\zeta (x)\nabla {{ \bm{{\cal U}}}_h^n},\nabla { \bm{\chi}_h}} \right)  - \frac{{\tau}}{2}{ \mathbfit{I}}\left( {{{ \bm{{\cal U}}}_h^n},\frac{\partial }{{\partial x}}{ \bm{\chi}_h}} \right)  - \frac{{\tau}}{2}{ \mathbfit{I}}\left( {{{ \bm{{\cal U}}}_h^n},\frac{\partial }{{\partial y}}{ \bm{\chi}_h}} \right)
+\left( W^n ,\textbf{V}
\right) \nonumber\\\nonumber
\\
 &=& {{ \mathbfit{H}}_2}\left( {{{ \bm{{\cal U}}}_h^{n - 1}},{ \bm{\chi}_h}} \right)  - \frac{{\tau}}{2}{ \mathbfit{I}}\left( {\zeta (x)\nabla {{ \bm{{\cal U}}}_h^{n - 1}},\nabla { \bm{\chi}_h}} \right)  + \frac{{\tau}}{2}{ \mathbfit{I}}\left( {{{ \bm{{\cal U}}}_h^{n - 1}},\frac{\partial }{{\partial x}}{ \bm{\chi}_h}} \right) \nonumber\\\nonumber
\\
 &+& \frac{{\tau}}{2}{ \mathbfit{I}}\left( {{{ \bm{{\cal U}}}_h^{n - 1}},\frac{\partial }{{\partial y}}{ \bm{\chi}}} \right)  - \tau{ \mathbfit{N}}\left( {{ \mathbfit{F}^{n-1}},{ \bm{\chi}_h}} \right)+\left( W^{n-1} ,\textbf{V}
 \right) \qquad\forall~{ \bm{\chi}_h} \in { \mathbfit{H}}_0^1(\Omega ).
\end{eqnarray}
\allowdisplaybreaks}

\begin{lemma}
Let
\begin{eqnarray}\label{LL1}
\left( {{\cal G}_{r,d}^n,{{ \bm{\chi}}_r}} \right)  &=& \left( {{\cal P}_h^1\widehat { \bm{{\cal U}}}_r^n - \widehat { \bm{{\cal U}}}_r^n\,,{{ \bm{\chi}}_r}} \right)  + \tau {{ \mathbfit{A}}_2}\left( {{\cal P}_h^1\widehat { \bm{{\cal U}}}_r^n - \widehat { \bm{{\cal U}}}_r^n\,,\frac{\partial }{{\partial x}}{{ \bm{\chi}}_r}} \right)  + \tau {{ \mathbfit{A}}_3}\left( {{\cal P}_h^1\widehat { \bm{{\cal U}}}_r^n - \widehat { \bm{{\cal U}}}_r^n\,,\frac{\partial }{{\partial y}}{{ \bm{\chi}}_r}} \right) \\\nonumber
\\\nonumber
 &-& \left( {{\cal P}_h^1\widehat { \bm{{\cal U}}}_r^{n - 1} - \widehat { \bm{{\cal U}}}_r^{n - 1}\,,{{ \bm{\chi}}_r}} \right)  + \tau {{ \mathbfit{A}}_2}\left( {{\cal P}_h^1\widehat { \bm{{\cal U}}}_r^{n - 1} - \widehat { \bm{{\cal U}}}_r^{n - 1},\frac{\partial }{{\partial x}}{{ \bm{\chi}}_r}} \right)  + \tau {{ \mathbfit{A}}_3}\left( {{\cal P}_h^1\widehat { \bm{{\cal U}}}_r^{n - 1} - \widehat { \bm{{\cal U}}}_r^{n - 1},\frac{\partial }{{\partial y}}{{ \bm{\chi}}_r}} \right) .
\end{eqnarray}
Then, we have
\[\left\| {{\cal G}_{r,d}^n} \right\|_{{L^2}(\Omega )}^{} \le CN_{}^{1 - \upsilon }.\]
\end{lemma}
\begin{proof}
Eq. \eqref{LL1} is changed to
\[\begin{array}{l}
\left( {{\cal G}_{r,d}^n,{{ \bm{\chi}}_r}} \right)  = \left( {{\cal P}_h^1\widehat { \bm{{\cal U}}}_r^n - \widehat { \bm{{\cal U}}}_r^n{\mkern 1mu} ,{{ \bm{\chi}}_r}} \right)  - \tau {{ \mathbfit{A}}_2}\left( {\displaystyle\frac{\partial }{{\partial x}}\left( {{\cal P}_h^1\widehat { \bm{{\cal U}}}_r^n - \widehat { \bm{{\cal U}}}_r^n} \right){\mkern 1mu} ,{{ \bm{\chi}}_r}} \right)  - \tau {{ \mathbfit{A}}_3}\left( {\displaystyle\frac{\partial }{{\partial y}}\left( {{\cal P}_h^1\widehat { \bm{{\cal U}}}_r^n - \widehat { \bm{{\cal U}}}_r^n} \right){\mkern 1mu} ,{{ \bm{\chi}}_r}} \right) \\
\\
 - \left( {{\cal P}_h^1\widehat { \bm{{\cal U}}}_r^{n - 1} - \widehat { \bm{{\cal U}}}_r^{n - 1}{\mkern 1mu} ,{{ \bm{\chi}}_r}} \right)  - \tau {{ \mathbfit{A}}_2}\left( {\displaystyle\frac{\partial }{{\partial x}}\left( {{\cal P}_h^1\widehat { \bm{{\cal U}}}_r^{n - 1} - \widehat { \bm{{\cal U}}}_r^{n - 1}} \right),{{ \bm{\chi}}_r}} \right)  - \tau {{ \mathbfit{A}}_3}\left( {\displaystyle\frac{\partial }{{\partial y}}\left( {{\cal P}_h^1\widehat { \bm{{\cal U}}}_r^{n - 1} - \widehat { \bm{{\cal U}}}_r^{n - 1}} \right),{{ \bm{\chi}}_r}} \right) .
\end{array}\]
From the above relation, by setting ${{ \bm{\chi}}_r} = Y_{r,d}^n$ we have
\[\begin{array}{l}
\left\| {{\cal G}_{r,d}^n} \right\|_{{L^2}(\Omega )} \le {\left\| {{\cal P}_h^1\widehat { \bm{{\cal U}}}_r^n - \widehat { \bm{{\cal U}}}_r^n} \right\|_{{L^2}(\Omega )}} + \tau {{ \mathbfit{A}}_2}{\left\| {\displaystyle\frac{\partial }{{\partial x}}\left( {{\cal P}_h^1\widehat { \bm{{\cal U}}}_r^n - \widehat { \bm{{\cal U}}}_r^n} \right)} \right\|_{{L^2}(\Omega )}} + \tau {{ \mathbfit{A}}_3}{\left\| {\displaystyle\frac{\partial }{{\partial y}}\left( {{\cal P}_h^1\widehat { \bm{{\cal U}}}_r^n - \widehat { \bm{{\cal U}}}_r^n} \right)} \right\|_{{L^2}(\Omega )}}\\
\\
 + {\left\| {{\cal P}_h^1\widehat { \bm{{\cal U}}}_r^{n - 1} - \widehat { \bm{{\cal U}}}_r^{n - 1}} \right\|_{{L^2}(\Omega )}} + \tau {{ \mathbfit{A}}_2}{\left\| {\displaystyle\frac{\partial }{{\partial x}}\left( {{\cal P}_h^1\widehat { \bm{{\cal U}}}_r^{n - 1} - \widehat { \bm{{\cal U}}}_r^{n - 1}} \right)} \right\|_{{L^2}(\Omega )}} + \tau {{ \mathbfit{A}}_3}{\left\| {\displaystyle\frac{\partial }{{\partial y}}\left( {{\cal P}_h^1\widehat { \bm{{\cal U}}}_r^{n - 1} - \widehat { \bm{{\cal U}}}_r^{n - 1}} \right)} \right\|_{{L^2}(\Omega )}},
\end{array}\]
which concludes the proof.
\end{proof}


\begin{Theorem}
Let ${  { \bm{{\cal U}}}_r^n}$ and ${{ \bm{{\cal U}}}_{h}^n}$  be solutions of (\ref{25P}) and  (\ref{26P}), respectively. Then  
\begin{equation}\label{E1}
{\left\|\mathbb{E}[ {\widehat { \bm{{\cal U}}}_r^n - { \bm{{\cal U}}}_{r,d}^n}] \right\|_{{L^2}(\Omega )}} \le C({\tau ^2} + N^{1-\nu}).
\end{equation}
\end{Theorem}
\begin{proof} Defining $\mathbfit{Z}^n:=\mathbb{E}[\textbf{u}^n-\textbf{U}_h^n]$
and subtracting \eqref{26P} from (\ref{25P}) give rise to
{\allowdisplaybreaks
\begin{align}\label{28P}\nonumber
& {{ \mathbfit{H}}_1}\left( \mathbfit{Z}^n,\textbf{v}  \right) \frac{{\tau}}{2}{ \mathbfit{I}}\left( {\zeta (x)\nabla {{ \mathbfit{Z}}^n} ,\nabla { \bm{\chi}}} \right)  
- \frac{{\tau}}{2}{ \mathbfit{I}}\left( {\mathbfit{Z}^n,\frac{\partial }{{\partial x}}{ \bm{\chi}}} \right)  - \frac{{\tau}}{2}{ \mathbfit{I}}\left( {\mathbfit{Z}^n,\frac{\partial }{{\partial y}}{ \bm{\chi}}} \right)=  \nonumber\\ \nonumber 
& \textbf{M}_2\left(  \mathbfit{Z}^{n-1},{ \bm{\chi}} \right)  - \frac{{\tau}}{2}{ \mathbfit{I}}\left( \zeta (x)\nabla  \mathbfit{Z}^{n-1}  \right) +
  \frac{{\tau}}{2}{ \mathbfit{I}}\left( \mathbfit{Z}^{n-1},\frac{\partial }{{\partial x}}{ \bm{\chi}} \right)\ 
 \nonumber\\\nonumber
 & + \frac{{\tau}}{2}{ \mathbfit{I}}\left( \mathbfit{Z}^{n-1} ,\frac{\partial }{{\partial y}}{ \bm{\chi}} \right) + \tau{ \mathbfit{N}}\left( {{ \mathbfit{F}^{n-1}}-{ \mathbfit{\overline F}^{n-1}},{ \bm{\chi}}} \right)+ \tau\left( {{ \mathbfit{R}_{t}^n},{ \bm{\chi}}} \right), \qquad\forall ~ { \bm{\chi}} \in { \mathbfit{H}}_0^1(\Omega ).
\end{align}
\allowdisplaybreaks} 
Then, we define $\bm{\varpi} _h^{1,n} := \mathbb{E} [{\mathbfit{P}}_h^1{{\bm{{\cal U}}}^n} - {\bm{{\cal U}}}_h^n]$ and $\bm{\eta} _{h}^{1,n}{\mkern 1mu}:= \mathbb{E} [{\bm{{\cal U}}}^n - {\mathbfit{P}}_h^1{{\bm{{\cal U}}}^n}]$, then
{\allowdisplaybreaks
\begin{eqnarray*}\label{30P}
{{\mathbfit{H}}_1}\left( {{\bm{\varpi}}_h^{1,n},{ \bm{\chi}}} \right) &+&\frac{{\tau}}{2}{\mathbfit{I}}\left( {\zeta ({ \mathbfit{x}})\nabla {\bm{\varpi}}_h^{1,n},\nabla { \bm{\chi}}} \right)  - \frac{{\tau}}{2}{\mathbfit{I}}\left( {{\bm{\varpi}}_h^{1,n},\frac{\partial }{{\partial x}}{ \bm{\chi}}} \right)  - \frac{{\tau}}{2}{ \mathbfit{I}}\left( {{\bm{\varpi}}_h^{1,n},\frac{\partial }{{\partial y}}{\bm{\chi}}} \right) \\
\\
 &=& {{\mathbfit{H}}_2}\left( {{\bm{\varpi}}_h^{1,n - 1},{\bm{\chi}}} \right)  - \frac{{\tau}}{2}{\mathbfit{I}}\left( {\zeta ({\mathbfit{x}})\nabla {\bm{\varpi}}_h^{1,n - 1},\nabla {\bm{\chi}}} \right) \\
\\
 &+& \frac{{\tau}}{2}{\mathbfit{I}}\left( {{\bm{\varpi}}_h^{1,n - 1},\frac{\partial }{{\partial x}}{\bm{\chi}}} \right)  + \frac{{\tau}}{2}{\mathbfit{I}}\left( {{\bm{\varpi}}_h^{1,n - 1},\frac{\partial }{{\partial y}}{\bm{\chi}}} \right) \\
\\
& -& \tau{\mathbfit{N}}\left( {{{\mathbfit{F}}^{n - 1}} - {{{\mathbfit{\bar F}}}^{n - 1}},{\bm{\chi}}} \right)  + \tau\left( {{\mathbfit{R}}_t^n,{ \bm{\chi}}} \right)  - {{\mathbfit{H}}_1}\left( {{\mathbfit{\Lambda }}_h^{1,n},{\bm{\chi}}} \right) {\mkern 1mu} \\
\\
 &+& \frac{{\tau}}{2}{\mathbfit{I}}\left( {{\mathbfit{\Lambda }}_h^{1,n},\frac{\partial }{{\partial x}}{\bm{\chi}}} \right)  + \frac{{\tau}}{2}{ \mathbfit{I}}\left( {{\mathbfit{\Lambda }}_h^{1,n},\frac{\partial }{{\partial y}}{\bm{\chi}}} \right)  + {{\mathbfit{H}}_2}\left( {{\mathbfit{\Lambda }}_h^{1,n - 1},{\bm{\chi}}} \right) \\
\\
& +& \frac{{\tau}}{2}{\mathbfit{I}}\left( {{\mathbfit{\Lambda }}_h^{1,n - 1},\frac{\partial }{{\partial x}}{\bm{\chi}}} \right)  + \frac{{\tau}}{2}{\mathbfit{I}}\left( {{\mathbfit{\Lambda }}_h^{1,n - 1},\frac{\partial }{{\partial y}}{\bm{\chi}}} \right) ,\,\,\,\,\,\,\,\,\forall {\mkern 1mu} \,{\mkern 1mu} {\bm{\chi}} \in {\mathbfit{H}}_0^1(\Omega ).
\end{eqnarray*}
\allowdisplaybreaks}
Thus, by assuming
\begin{eqnarray*}\label{31P}
\left( {{\mathbfit{\Phi }}_h^{1,n},{\bm{\chi}}} \right)  &=&  - {{\mathbfit{H}}_1}\left( {{\mathbfit{\Lambda }}_h^{1,n},{\bm{\chi}}} \right) {\mkern 1mu}  + \frac{{\tau}}{2}{\mathbfit{I}}\left( {{\mathbfit{\Lambda }}_h^{1,n},\frac{\partial }{{\partial x}}{\bm{\chi}}} \right)  + \frac{{\tau}}{2}{\mathbfit{I}}\left( {{\mathbfit{\Lambda }}_h^{1,n},\frac{\partial }{{\partial y}}{\bm{\chi}}} \right) \\
\\
 &+& {{\mathbfit{H}}_2}\left( {{\mathbfit{\Lambda }}_h^{1,n - 1},{\bm{\chi}}} \right)  + \frac{{\tau}}{2}{\mathbfit{I}}\left( {{\mathbfit{\Lambda }}_h^{1,n - 1},\frac{\partial }{{\partial x}}{\bm{\chi}}} \right)  + \frac{{\tau}}{2}{\mathbfit{I}}\left( {{\mathbfit{\Lambda }}_h^{1,n - 1},\frac{\partial }{{\partial y}}{\bm{\chi}}} \right) ,
\end{eqnarray*}
we have
\begin{eqnarray*}\label{32P}
{{\mathbfit{H}}_1}\left( {{\bm{\varpi}}_h^{1,n},{\bm{\chi}}} \right) &+&\frac{{\tau}}{2}{\mathbfit{I}}\left( {\zeta ({\mathbfit{x}})\nabla {\bm{\varpi}}_h^{1,n},\nabla {\bm{\chi}}} \right)  - \frac{{\tau}}{2}{\mathbfit{I}}\left( {{ \bm{\varpi}}_h^{1,n},\frac{\partial }{{\partial x}}{ \bm{\chi}}} \right)  - \frac{{\tau}}{2}{ \mathbfit{I}}\left( {{ \bm{\varpi}}_h^{1,n},\frac{\partial }{{\partial y}}{ \bm{\chi}}} \right) \\
\\
 &=& {{ \mathbfit{H}}_2}\left( {{ \bm{\varpi}}_h^{1,n - 1},{ \bm{\chi}}} \right)  - \frac{{\tau}}{2}{ \mathbfit{I}}\left( {\zeta ({ \mathbfit{x}})\nabla { \bm{\varpi}}_h^{1,n - 1},\nabla { \bm{\chi}}} \right)  + \frac{{\tau}}{2}{ \mathbfit{I}}\left( {{ \bm{\varpi}}_h^{1,n - 1},\frac{\partial }{{\partial x}}{ \bm{\chi}}} \right) \\
\\
 &+& \frac{{\tau}}{2}{ \mathbfit{I}}\left( {{ \bm{\varpi}}_h^{1,n - 1},\frac{\partial }{{\partial y}}{ \bm{\chi}}} \right) - \tau{ \mathbfit{N}}\left( {{{ \mathbfit{F}}^{n - 1}} - {{{ \mathbfit{\bar F}}}^{n - 1}},{ \bm{\chi}}} \right)  + \tau\left( {{ \mathbfit{R}}_t^n,{ \bm{\chi}}} \right)  + \left( {{ \mathbfit{\Phi }}_h^{1,n},{ \bm{\chi}}} \right),\,\,\,\forall {\mkern 1mu} \,{\mkern 1mu} { \bm{\chi}} \in { \mathbfit{H}}_0^1(\Omega ).
\end{eqnarray*}
Setting ${{ \bm{\chi}}_r}=\rm X _{r,d}^n$, gives
\begin{eqnarray*}\label{33P}
{{ \mathbfit{H}}_1}\left( {{ \bm{\varpi}}_h^{1,n},{ \bm{\varpi}}_h^{1,n}} \right) {\rm{ + }}\frac{{\tau}}{2}{ \mathbfit{I}}\left( {\zeta ({ \mathbfit{x}})\nabla { \bm{\varpi}}_h^{1,n},\nabla { \bm{\varpi}}_h^{1,n}} \right)  - \frac{{\tau}}{2}{ \mathbfit{I}}\left( {{ \bm{\varpi}}_h^{1,n},\frac{\partial }{{\partial x}}{ \bm{\varpi}}_h^{1,n}} \right)  - \frac{{\tau}}{2}{ \mathbfit{I}}\left( {{ \bm{\varpi}}_h^{1,n},\frac{\partial }{{\partial y}}{ \bm{\varpi}}_h^{1,n}} \right) \\
\\
 = {{ \mathbfit{H}}_2}\left( {{ \bm{\varpi}}_h^{1,n - 1},{ \bm{\varpi}}_h^{1,n}} \right)  - \frac{{\tau}}{2}{ \mathbfit{I}}\left( {\zeta ({ \mathbfit{x}})\nabla { \bm{\varpi}}_h^{1,n - 1},\nabla { \bm{\varpi}}_h^{1,n}} \right)  + \frac{{\tau}}{2}{ \mathbfit{I}}\left( {{ \bm{\varpi}}_h^{1,n - 1},\frac{\partial }{{\partial x}}{ \bm{\varpi}}_h^{1,n}} \right) \\
\\
 + \frac{{\tau}}{2}{ \mathbfit{I}}\left( {{ \bm{\varpi}}_h^{1,n - 1},\frac{\partial }{{\partial y}}{ \bm{\varpi}}_h^{1,n}} \right)  - \tau{ \mathbfit{N}}\left( {{{ \mathbfit{F}}^{n - 1}} - {{{ \mathbfit{\bar F}}}^{n - 1}},{ \bm{\varpi}}_h^{1,n}} \right)  + \tau\left( {{ \mathbfit{R}}_t^n,{ \bm{\varpi}}_h^{1,n}} \right)  + \left( {{ \mathbfit{\Phi }}_h^{1,n},{ \bm{\varpi}}_h^{1,n}} \right) 
\end{eqnarray*}
Thus, we can write
{\allowdisplaybreaks
\begin{eqnarray*}\label{34P}
\left\| {{{ \mathbfit{H}}_1}} \right\|\left\| {{ \bm{\varpi}}_h^{1,n}} \right\|_{{L^2}(\Omega )}^2&+&\frac{{\tau}}{2}{ \mathbfit{I}}\left\| {\zeta ({ \mathbfit{x}})} \right\|\left\| {\nabla { \bm{\varpi}}_h^{1,n}} \right\|_{{L^2}(\Omega )}^2\\
\\
& \le& \frac{{\tau}}{2}{ \mathbfit{I}}\left( {{ \bm{\varpi}}_h^{1,n},\frac{\partial }{{\partial x}}{ \bm{\varpi}}_h^{1,n}} \right)  + \frac{{\tau}}{2}{ \mathbfit{I}}\left( {{ \bm{\varpi}}_h^{1,n},\frac{\partial }{{\partial y}}{ \bm{\varpi}}_h^{1,n}} \right) \\
\\
 &+& \left\| {{{ \mathbfit{H}}_2}} \right\|\left\| {{ \bm{\varpi}}_h^{1,n - 1}} \right\|\left\| {{ \bm{\varpi}}_h^{1,n}} \right\| + \frac{{\tau}}{2}{ \mathbfit{I}}\left\| {\zeta ({ \mathbfit{x}})} \right\|\left\| {\nabla { \bm{\varpi}}_h^{1,n - 1}} \right\|\left\| {\nabla { \bm{\varpi}}_h^{1,n}} \right\|\\
\\
 &+& \frac{{\tau}}{2}{ \mathbfit{I}}\left( {{ \bm{\varpi}}_h^{1,n - 1},\frac{\partial }{{\partial x}}{ \bm{\varpi}}_h^{1,n}} \right)  + \frac{{\tau}}{2}{ \mathbfit{I}}\left( {{ \bm{\varpi}}_h^{1,n - 1},\frac{\partial }{{\partial y}}{ \bm{\varpi}}_h^{1,n}} \right) \\
\\
 &-& \tau{ \mathbfit{LN}}\left\| {{ \bm{\varpi}}_h^{1,n - 1}} \right\|\left\| {{ \bm{\varpi}}_h^{1,n}} \right\| + \tau\left\| {{ \mathbfit{R}}_t^n} \right\|\left\| {{ \bm{\varpi}}_h^{1,n}} \right\| + \tau\left\| {{ \mathbfit{\Phi }}_h^{1,n}} \right\|\left\| {{ \bm{\varpi}}_h^{1,n}} \right\|.
\end{eqnarray*}
\allowdisplaybreaks}
Also, let
\begin{equation}
2\varpi \le \left\| {{{ \mathbfit{A}}_1}} \right\| \le \varpi,\,\,\,\,\,\,\,\,\,\,\,\,\,\,\,\,\,\left\| {{{ \mathbfit{A}}_2}} \right\|,\left\| {{{ \mathbfit{A}}_3}} \right\| \le C .
\end{equation}
As a result
{\allowdisplaybreaks
\begin{eqnarray}\label{35P}
\frac{1}{2}\left\| {{ \bm{\varpi}}_h^{1,n}} \right\|_{{{ \bm{L}}^2}(\Omega )}^2 &+& \frac{{\tau}}{4}\left\| {\zeta ({ \mathbfit{x}})} \right\|\left\| {\nabla {{ \bm{\varpi}}_h^{1,n}}} \right\|_{{{ \bm{L}}^2}(\Omega )}^2 \\ \nonumber
\\
&\le& \frac{1}{2}\left\| {{{ \mathbfit{H}}_1}} \right\|\left\| {{ \bm{\varpi}}_h^{1,{n-1}}} \right\|_{{{ \bm{L}}^2}(\Omega )}^2 + \frac{{\tau}}{2}\left\| {\zeta ({ \mathbfit{x}})} \right\|\left\| {\nabla {{ \bm{\varpi}}_h^{1,{n-1}}}} \right\|_{{{ \bm{L}}^2}(\Omega )}^2\nonumber\\\nonumber
\\\nonumber
&+&\frac{{C_1 \bm{L}\tau}}{{2\left\| {\zeta ({ \mathbfit{x}})} \right\|}}\left\| {{ \bm{\varpi}}_h^{1,n}}\right\|_{{{ \bm{L}}^2}(\Omega )}^2 + \frac{{C_2 \bm{L}\tau}}{{2\left\| {\zeta ({ \mathbfit{x}})} \right\|}}\left\|{{ \bm{\varpi}}_h^{1,{n-1}}} \right\|_{{{ \bm{L}}^2}(\Omega )}^2\nonumber\\\nonumber
\\\nonumber
&+&C\tau\left\| {{ \mathbfit{R}}_t^n} \right\|_{{L^2}(\Omega )}^2 + C\tau\left\| {{ \bm{\varpi}}_h^{1,n}} \right\|_{{L^2}(\Omega )}^2 + C\tau\left\| {{ \mathbfit{\Phi }}_h^{1,n}} \right\|_{{L^2}(\Omega )}^2 + C\tau\left\| {{ \bm{\varpi}}_h^{1,n}} \right\|_{{L^2}(\Omega )}^2.
\end{eqnarray}
\allowdisplaybreaks}
Applying the  definition
\[\left\| {{ \bm{\varpi}}_h^{1,n}} \right\|_{{{ \mathbfit{H}}_w}(\Omega )}^2 := \left\| {{ \bm{\varpi}}_h^{1,n}} \right\|_{{{ \bm{L}}^2}(\Omega )}^2 + \frac{1}{2}\tau \varpi\left\| {\nabla {{ \bm{\varpi}}_h^{1,n}}} \right\|_{{{ \bm{L}}^2}(\Omega )}^2,\]
Eq. \eqref{35P} can be written as
{\allowdisplaybreaks
\begin{eqnarray}\label{36P}
 \left\| {{ \bm{\varpi}}_h^{1,n}} \right\|_{{{ \bm{L}}^2}(\Omega )}^2 &+& \frac{{\tau}}{2}\left\| {\zeta ({ \mathbfit{x}})} \right\|\left\| {\nabla {{ \bm{\varpi}}_h^{1,n}}} \right\|_{{{ \bm{L}}^2}(\Omega )}^2 \\ \nonumber
\\
&\le& \left\| {{{ \mathbfit{H}}_1}} \right\|\left\| {{ \bm{\varpi}}_h^{1,{n-1}}} \right\|_{{{ \bm{L}}^2}(\Omega )}^2 + \frac{{\tau}}{2}\left\| {\zeta ({ \mathbfit{x}})} \right\|\left\| {\nabla {{ \bm{\varpi}}_h^{1,{n-1}}}} \right\|_{{{ \bm{L}}^2}(\Omega )}^2\nonumber\\\nonumber
\\\nonumber
&+&\frac{{C_1^* \bm{L}\tau}}{{2\left\| {\zeta ({ \mathbfit{x}})} \right\|}}\left\| {{ \bm{\varpi}}_h^{1,n}}\right\|_{{{ \bm{L}}^2}(\Omega )}^2 + \frac{{C_2^* \bm{L}\tau}}{{2\left\| {\zeta ({ \mathbfit{x}})} \right\|}}\left\|{{ \bm{\varpi}}_h^{1,{n-1}}} \right\|_{{{ \bm{L}}^2}(\Omega )}^2\nonumber\\\nonumber
\\\nonumber
&+&C\tau\left\| {{ \mathbfit{R}}_t^n} \right\|_{{L^2}(\Omega )}^2  + C\tau\left\| {{ \mathbfit{\Phi }}_h^{1,n}} \right\|_{{L^2}(\Omega )}^2.
\end{eqnarray}
\allowdisplaybreaks}
Now, from the above Eq., we have
{\allowdisplaybreaks
\begin{eqnarray*}\label{37P}
\sum\limits_{m = 1}^n {\left\| {{ \bm{\varpi}}_h^{1,m}} \right\|_{{L^2}(\Omega )}^2}  &+& \frac{{\tau}}{2}\left\| {\zeta ({ \mathbfit{x}})} \right\|\sum\limits_{m = 1}^n {\left\| {\nabla { \bm{\varpi}}_h^{1,m}} \right\|_{{L^2}(\Omega )}^2} \\
\\
 &\le &\sum\limits_{m = 1}^n {\left\| {{ \bm{\varpi}}_h^{1,m - 1}} \right\|_{{L^2}(\Omega )}^2}  + \frac{{\tau}}{2}\left\| {\zeta ({ \mathbfit{x}})} \right\|\sum\limits_{m = 1}^n {\left\| {\nabla { \bm{\varpi}}_h^{1,m - 1}} \right\|_{{L^2}(\Omega )}^2} \\
\\
 &+& \frac{{C_1^*{ \bm{L}}\tau}}{{2\left\| {\zeta ({ \mathbfit{x}})} \right\|}}\sum\limits_{m = 1}^n {\left\| {{ \bm{\varpi}}_h^{1,m}} \right\|_{{L^2}(\Omega )}^2}  + \frac{{C_2^*{ \bm{L}}\tau}}{{2\left\| {\zeta ({ \mathbfit{x}})} \right\|}}\sum\limits_{m = 1}^n {\left\| {{ \bm{\varpi}}_h^{1,m - 1}} \right\|_{{L^2}(\Omega )}^2} \\
\\
 &+& C\tau\sum\limits_{m = 1}^n {\left\| {{ \mathbfit{R}}_t^m} \right\|_{{L^2}(\Omega )}^2}  + C\tau\sum\limits_{m = 1}^n {\left\| {{ \mathbfit{\Phi }}_h^{1,m}} \right\|_{{L^2}(\Omega )}^2} .
\end{eqnarray*}
\allowdisplaybreaks}
By engaging the Gronwall lemma, the above relation  can be  rewritten as
{\allowdisplaybreaks
\begin{eqnarray*}
\left\| {{ \bm{\varpi}}_h^{1,n}} \right\|_{{H_\omega }(\Omega )}^2 &\le& \frac{{C{ \bm{L}}\tau}}{{\left\| {\zeta ({ \mathbfit{x}})} \right\|}}\sum\limits_{m = 1}^n {\left\| {{ \bm{\varpi}}_h^{1,m}} \right\|_{{L^2}(\Omega )}^2}  + C\tau\sum\limits_{m = 1}^n {\left\| {{ \mathbfit{R}}_t^m} \right\|_{{L^2}(\Omega )}^2}  + C\tau\sum\limits_{m = 1}^n {\left\| {{ \mathbfit{\Phi }}_h^{1,m}} \right\|_{{L^2}(\Omega )}^2} \\
\\
 &\le& \frac{{C{ \bm{L}}\tau}}{{\left\| {\zeta ({ \mathbfit{x}})} \right\|}}\sum\limits_{m = 1}^n {\left\| {{ \bm{\varpi}}_h^{1,m}} \right\|_{{L^2}(\Omega )}^2}  + Cn\tau\left\| {{ \mathbfit{R}}_t^m} \right\|_{{L^2}(\Omega )}^2 + Cn\tau\left\| {{ \mathbfit{\Phi }}_h^{1,m}} \right\|_{{L^2}(\Omega )}^2\\
\\
& \le& \left[ {Cn\tau\left\| {{ \mathbfit{R}}_t^m} \right\|_{{L^2}(\Omega )}^2 + Cn\tau\left\| {{ \mathbfit{\Phi }}_h^{1,m}} \right\|_{{L^2}(\Omega )}^2} \right]\exp \left( {\frac{{C{ \bm{L}}n\tau}}{{\left\| {\zeta ({ \mathbfit{x}})} \right\|}}} \right)\\
\\
& \le& \left[ {CT\tau^2 + Cn\tau\left( {\tau^2+ {N ^{1-\nu}}} \right)} \right]\exp \left( {\frac{{C{ \bm{L}}T}}{{\left\| {\zeta ({ \mathbfit{x}})} \right\|}}} \right)\\
\\
& \le& C{\left( {\tau^2 + {N ^{1-\nu}}} \right)^2}.
\end{eqnarray*}
\allowdisplaybreaks}
Thus we have
\begin{equation}
\left\| {{ \bm{\varpi}}_h^{1,n}} \right\|_{{{ \bm{L}}^2}(\Omega )}^{} \le \left\| {{ \bm{\varpi}}_h^{1,n}} \right\|_{{{ \mathbfit{H}}_w}(\Omega )}^{} \le C({\tau ^2} + {N ^{1-\nu}}).
\end{equation}
\end{proof}


\section{Numerical discussions}
\label{section4}
Here, we provide two numerical examples to check the accuracy and the efficiency of the proposed numerical procedure. In both cases, in order to estimate the expected value $M=1000$ independent random variables are used.
\subsection{Test problem 1}\label{E1}
For the first example, we study the following numerical example  
with $\Omega=[0,1]\times [0,1]$ as follows \cite{W.Liu,ilati2016remediation}
\begin{equation}
\left\{ {\begin{array}{*{20}{l}}
	{\displaystyle\frac{{\partial u}}{{\partial t}} + \left( {\frac{{\partial u}}{{\partial x}} + \frac{{\partial u}}{{\partial y}}} \right) - D\left( {\frac{{{\partial ^2}u}}{{\partial {x^2}}} + \frac{{{\partial ^2}u}}{{\partial {y^2}}}} \right) + 0.6{\varpi _p}\frac{{uv}}{{\left( {1 + u} \right)\left( {v + 2} \right)}} = f(x,y,t)+dW,}\\
	{}\\
	{\displaystyle\frac{{\partial v}}{{\partial t}} + \left( {\frac{{\partial v}}{{\partial x}} + \frac{{\partial v}}{{\partial y}}} \right) - D\left( {\frac{{{\partial ^2}v}}{{\partial {x^2}}} + \frac{{{\partial ^2}v}}{{\partial {y^2}}}} \right) + 0.6{\varpi _p}\frac{{uv}}{{\left( {1 + u} \right)\left( {v + 2} \right)}} = g(x,y,t)+dW,}\\
	{}\\
	{\displaystyle\frac{{\partial w}}{{\partial t}} + \left( {\frac{{\partial w}}{{\partial x}} + \frac{{\partial w}}{{\partial y}}} \right) - D\left( {\frac{{{\partial ^2}w}}{{\partial {x^2}}} + \frac{{{\partial ^2}w}}{{\partial {y^2}}}} \right) + 0.6{\varpi _p}\frac{{uv}}{{\left( {1 + u} \right)\left( {v + 2} \right)}} + 2w = h(x,y,t)+dW,}
	\end{array}} \right.
\end{equation}
\begin{figure}[h!]
	\begin{center}
		\includegraphics[width=8.5cm, height=6cm]{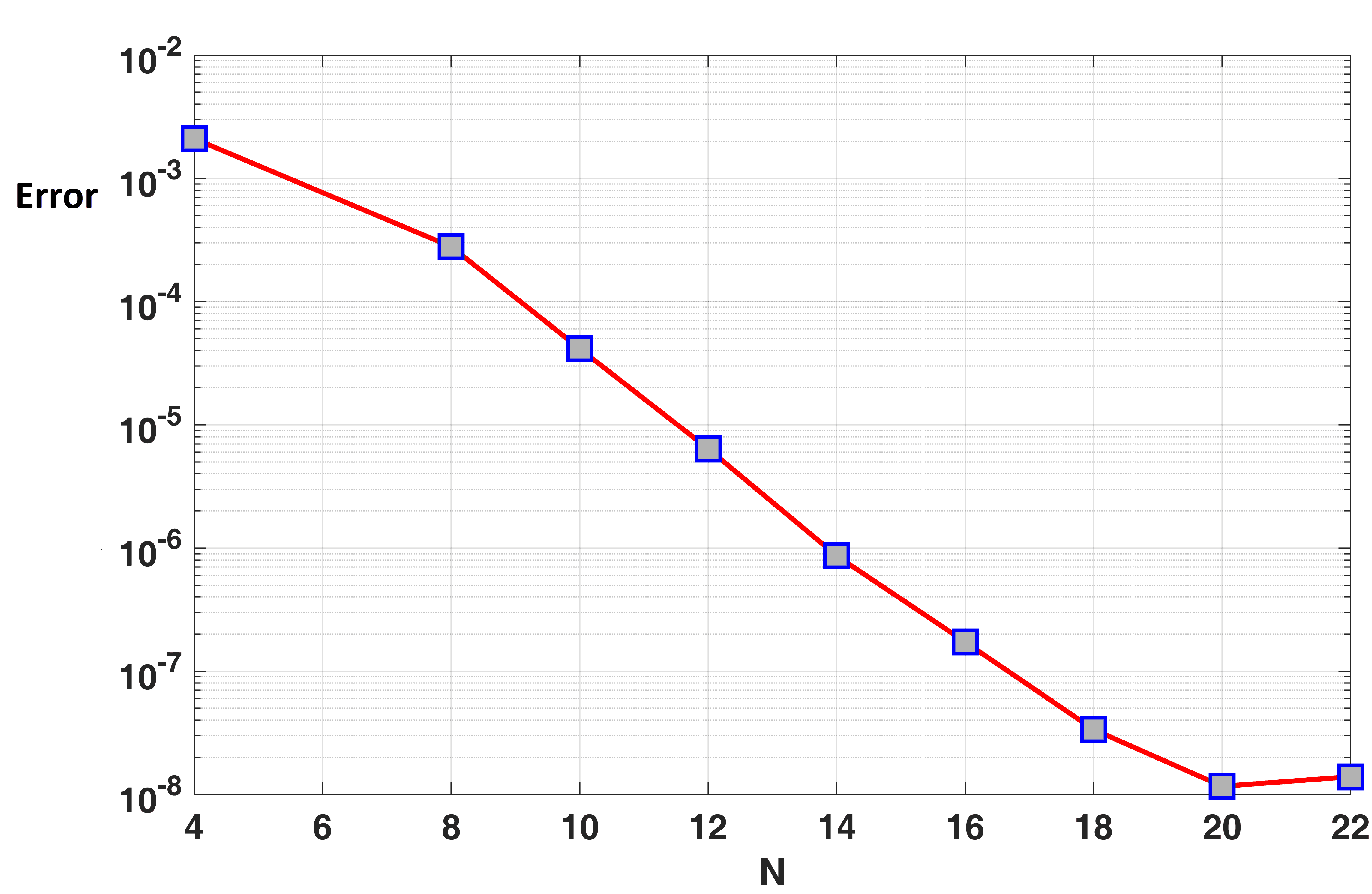}
		\includegraphics[width=8.5cm, height=6cm]{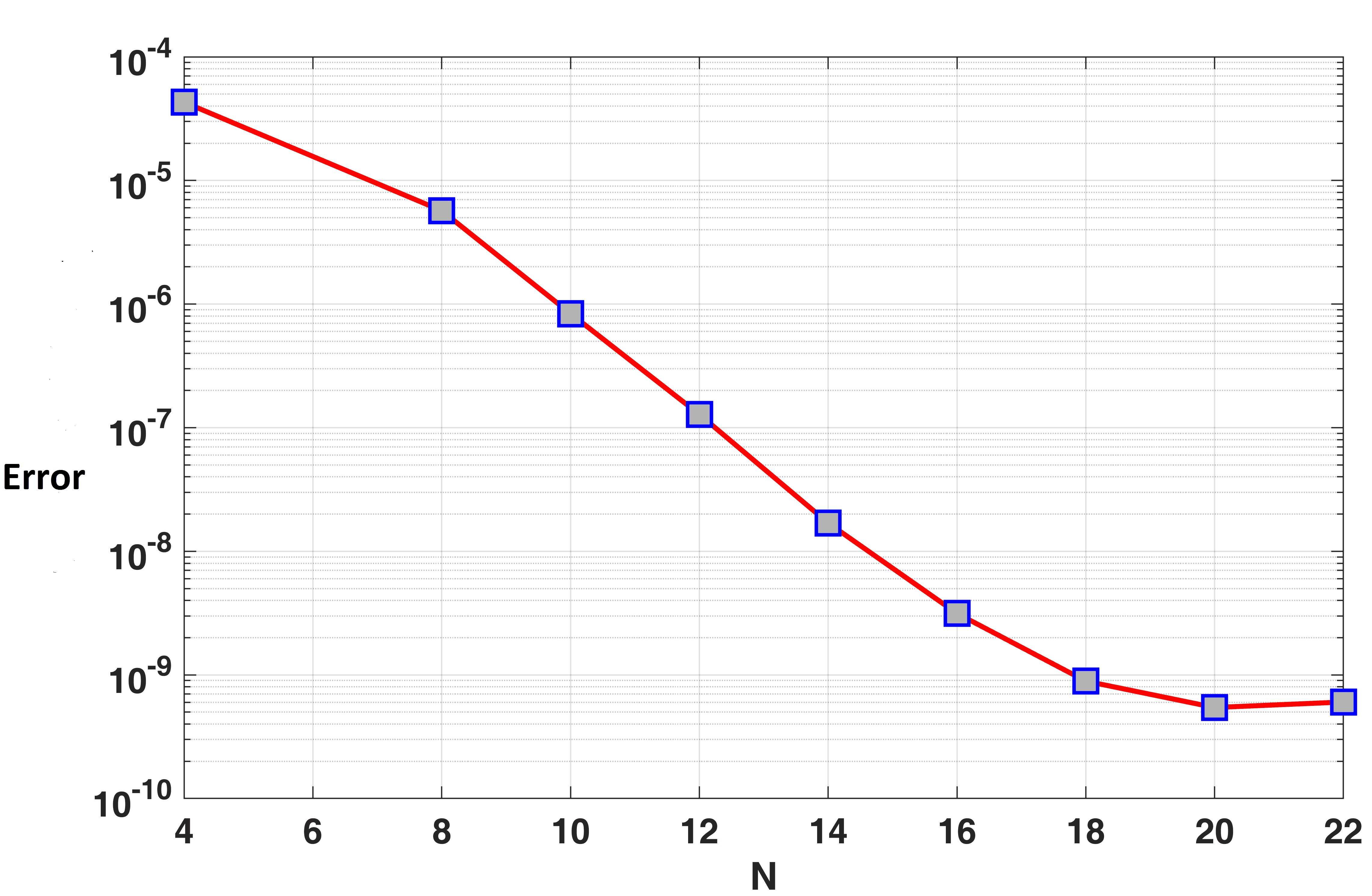}
		\caption{The computational error of expected value of the solution as a function of different number of basis functions (left panel $\tau=10^{-3}$ and right panel $\tau=10^{-4}$) for test problem 1.}\label{Fig1}
	\end{center}
\end{figure}

\begin{figure}[h!]
	\begin{center}
		\includegraphics[width=8.5cm, height=6cm]{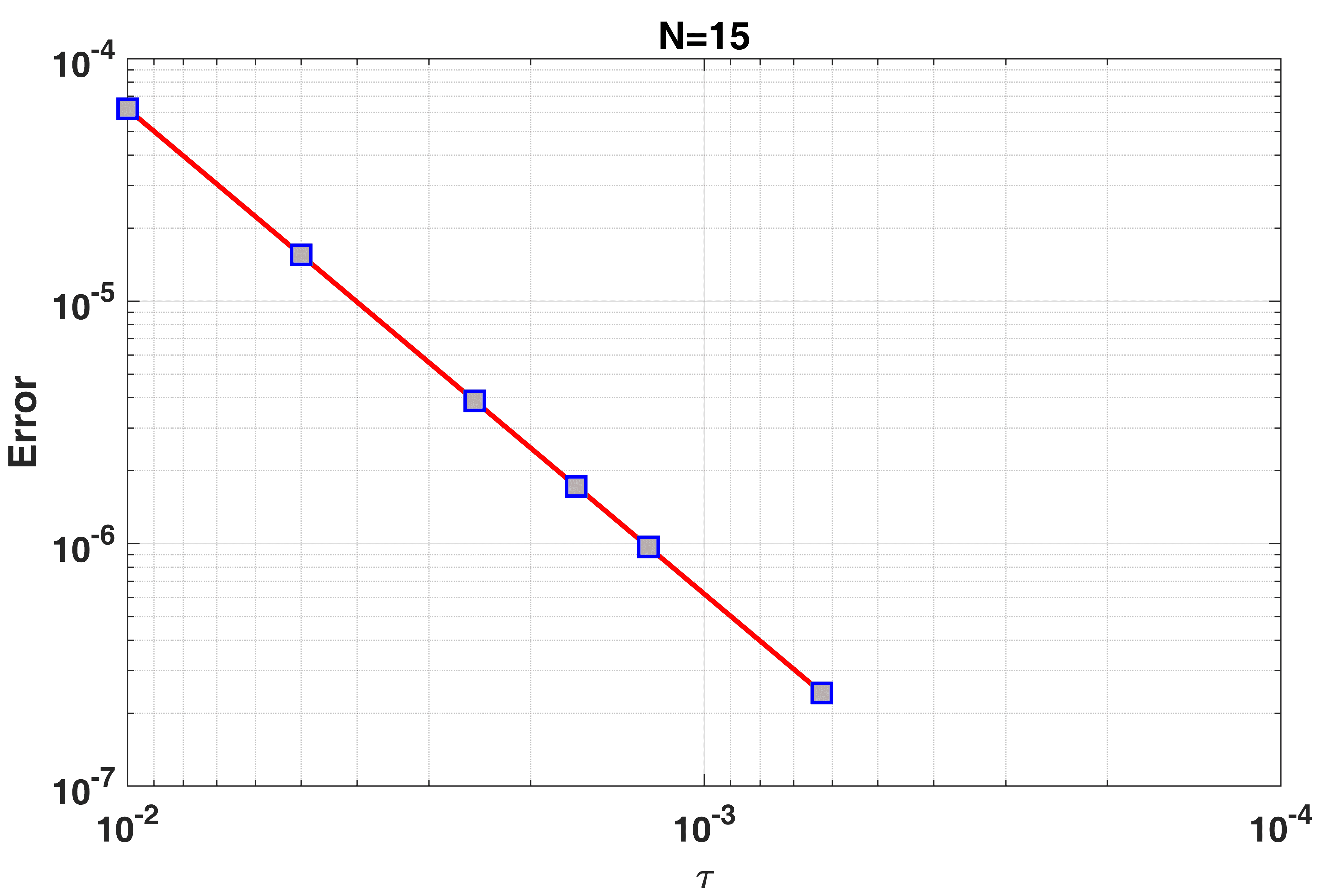}
		\includegraphics[width=8.5cm, height=6cm]{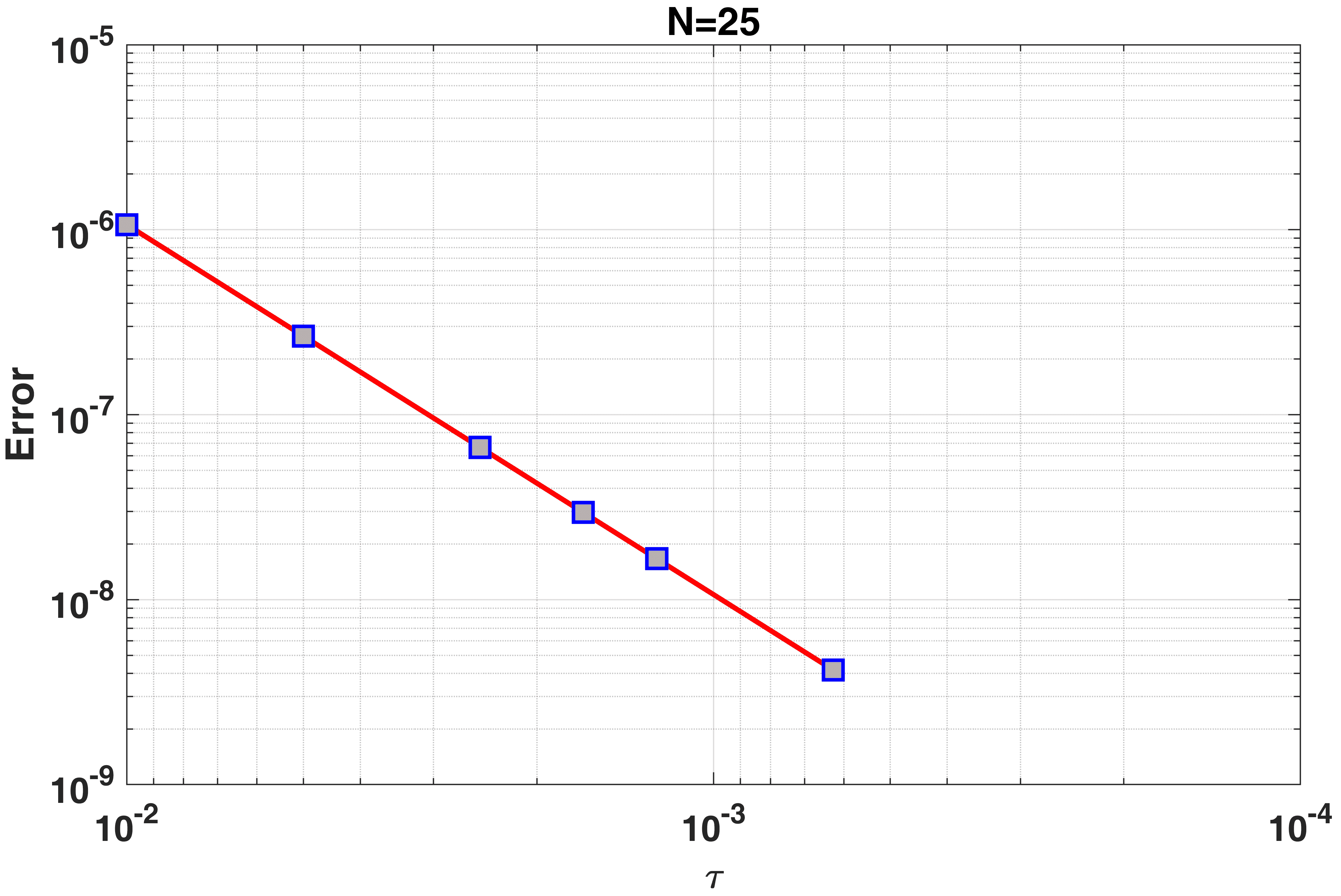}
		\caption{The computational error of expected value of the solution as a function of different number of basis functions (left panel $N=15$ and right panel $N=25$) for test problem 1.}\label{Fig10}
	\end{center}
\end{figure}

where the diffusion coefficient is $D={{10}^{ - 3}}$ and zero Dirichlet boundary conditions are applied. The initial conditions are
\begin{align}
u_0 = v_0 = w_0 = \sin (\pi x)\sin (\pi y). 
\end{align}
 We assume that the right hand sides are
\begin{flalign*}
f(x,y,t) &= (\pi -5)\text{e}^{-5t} \cos(\pi x) \sin(\pi y)\\
&+ \pi \text{e}^{-5t} \cos \left( {\pi y} \right)\sin \left( {\pi x} \right) + 2D{\pi ^2}\text{e}^{-5t}\sin(\pi x) \sin(\pi y)\\
& + 3\text{e}^{-10t} \sin(\pi x)^3  \sin(\pi y)^3 \left[  5 \left( \text{e}^{-2t} \sin(\pi x) \sin(\pi y)+2 \right) \left( \text{e}^{-5t}\sin(\pi x) \sin(\pi y)+1\right) \right]^{-1}.
\end{flalign*}
\begin{flalign*}
~~~\,g(x,y,t) &= \ee^{-2t} \left[ {\pi \cos \left( {\pi x} \right)\sin \left( {\pi y} \right) - 2\ee^{-2t}\sin \left( {\pi x} \right)\sin \left( {\pi y} \right)} +\pi \cos \left( {\pi y} \right)\sin \left( {\pi x} \right) + 2D{\pi ^2}\sin \left( {\pi {\rm{x}}} \right)u \sin(\pi y)\right]\\
&+ \ee^{-10t}\sin {\left( {\pi x} \right)^3}\sin {\left( {\pi y} \right)^3}{\left[ {10\left( {\ee^{-2t}\sin \left( {\pi x} \right)\sin \left( {\pi y} \right) + 2} \right)\left( {\ee^{-5t}\sin \left( {\pi x} \right)\sin \left( {\pi y} \right) + 1} \right)} \right]^{ - 1}}.
\end{flalign*}
\begin{flalign*}
h(x,y,t) &= \left( {\pi  - 1} \right)\ee^{-3t}\cos \left( {\pi x} \right)\sin \left( {\pi y} \right) + \pi \ee^{-3t}\cos \left( {\pi y} \right)\sin \left( {\pi x} \right)\\
&+ \ee^{-3t}\sin \left( {\pi x} \right)\left( {\pi \cos \left( {\pi y} \right) + 2D{\pi ^2}\sin \left( {\pi y} \right)} \right)\\
&+ 4{{\rm{e}}^{-10t}}\sin {\left( {\pi x} \right)^3}\sin {\left( {\pi y} \right)^3}{\left[ {5\left( {{{\rm{e}}^{ - 2t}}\sin \left( {\pi x} \right)\sin \left( {\pi y} \right) + 2} \right)\left( {{{\rm{e}}^{-5t}}\sin \left( {\pi x} \right)\sin \left( {\pi y} \right) + 1} \right)} \right]^{ - 1}}.
\end{flalign*}

In the deterministic case, the exact solution is
\[u(x,y,t) = \exp ( - 5t)\rho (x,y),{\mkern 1mu} {\mkern 1mu} \,\,\,\,\,v(x,y,t) = \exp ( - 2t)\rho (x,y),{\mkern 1mu} \,\,\,\,\,w(x,y,t) = \exp ( - 3t)\rho (x,y),\]
where $\rho (x,y) = \sin (\pi x)\sin (\pi y)$. In order to estimate the computational error, we use the reference solution with $N=30$ basis function. The developed LSEM method is used to approximate the expected value of the solution. In this example, we consider the summations of three computational errors with respect to $u$, $v$, and $w$ at $T=1$   where the results are shown in  
 Figure \ref{Fig1}  for different numbers of basis functions. As shown a noticeable error reduction has been achieved which indicates the method efficiency.
  We also estimated the solution for two different time steps, i.e., $\tau=10^{-3}$ and $\tau=10^{-4}$. 
  The computational error of expected value of the solution as a function of different number of basis functions (left panel $N=15$ and right panel $N=25$) has been depicted in  Figure \ref{Fig10} for test problem 1. The results show that as we expected smaller time steps gives rises to better error convergence.

\begin{figure}[h!]
	\begin{center}
		\includegraphics[width=8.5cm, height=6cm]{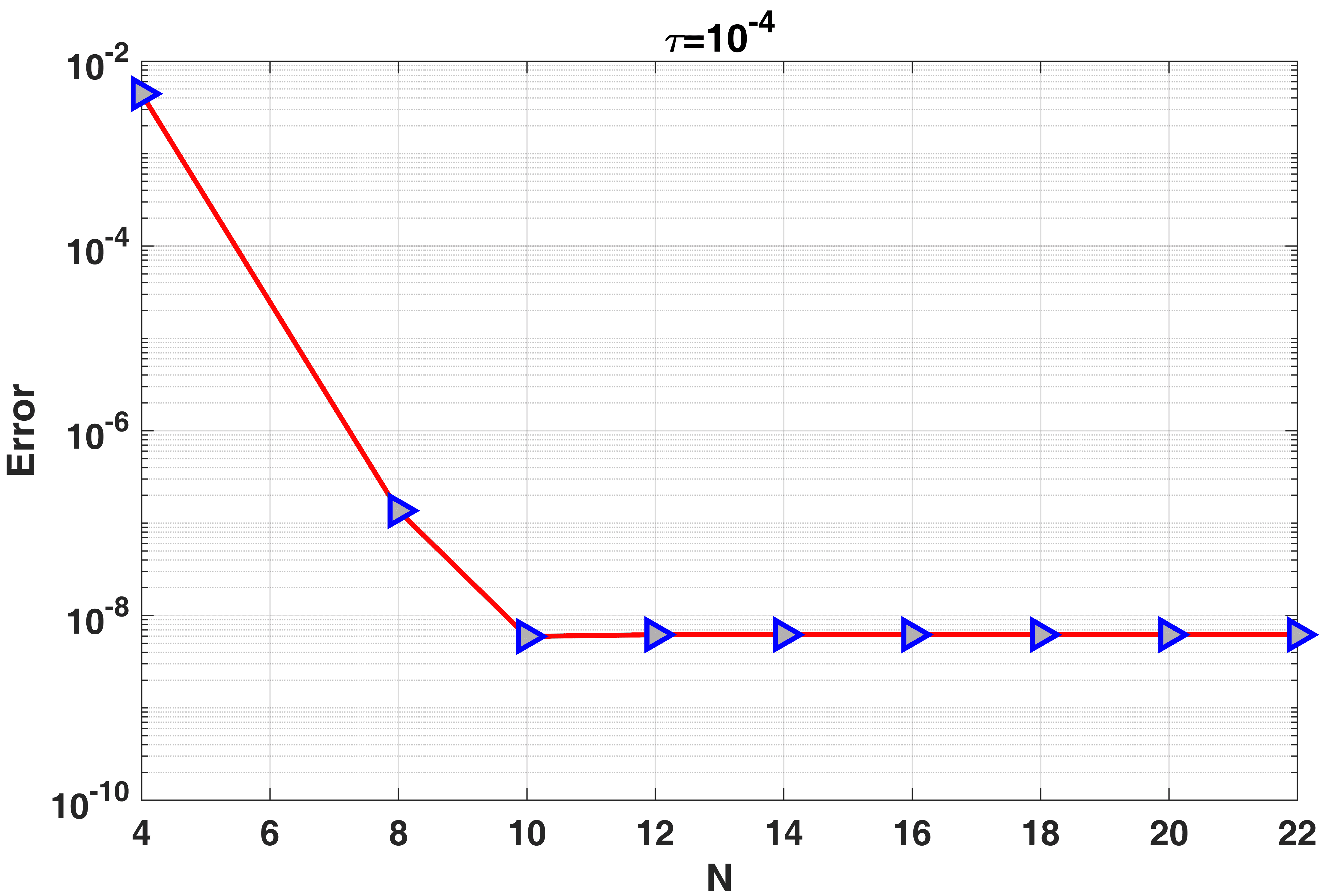}
		\includegraphics[width=8.5cm, height=6cm]{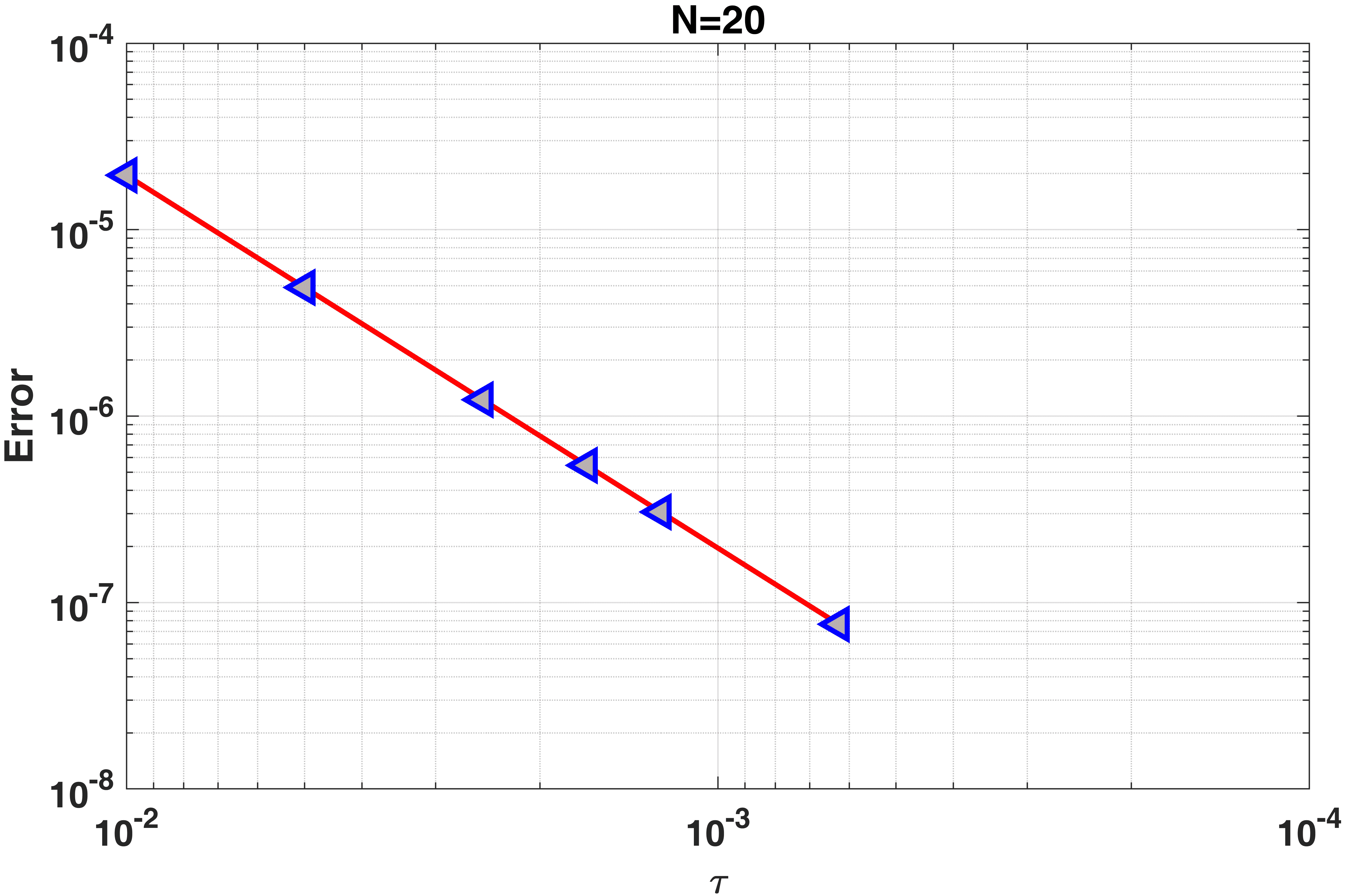}
		\caption{The computational error of expected value of the solution as a function of different number of basis functions  for test problem 2.}\label{Fig12}
	\end{center}
\end{figure}

\begin{figure}[h!]
	\begin{center}
		\vspace{-1.cm}
		\includegraphics[width=8.5cm, height=7cm]{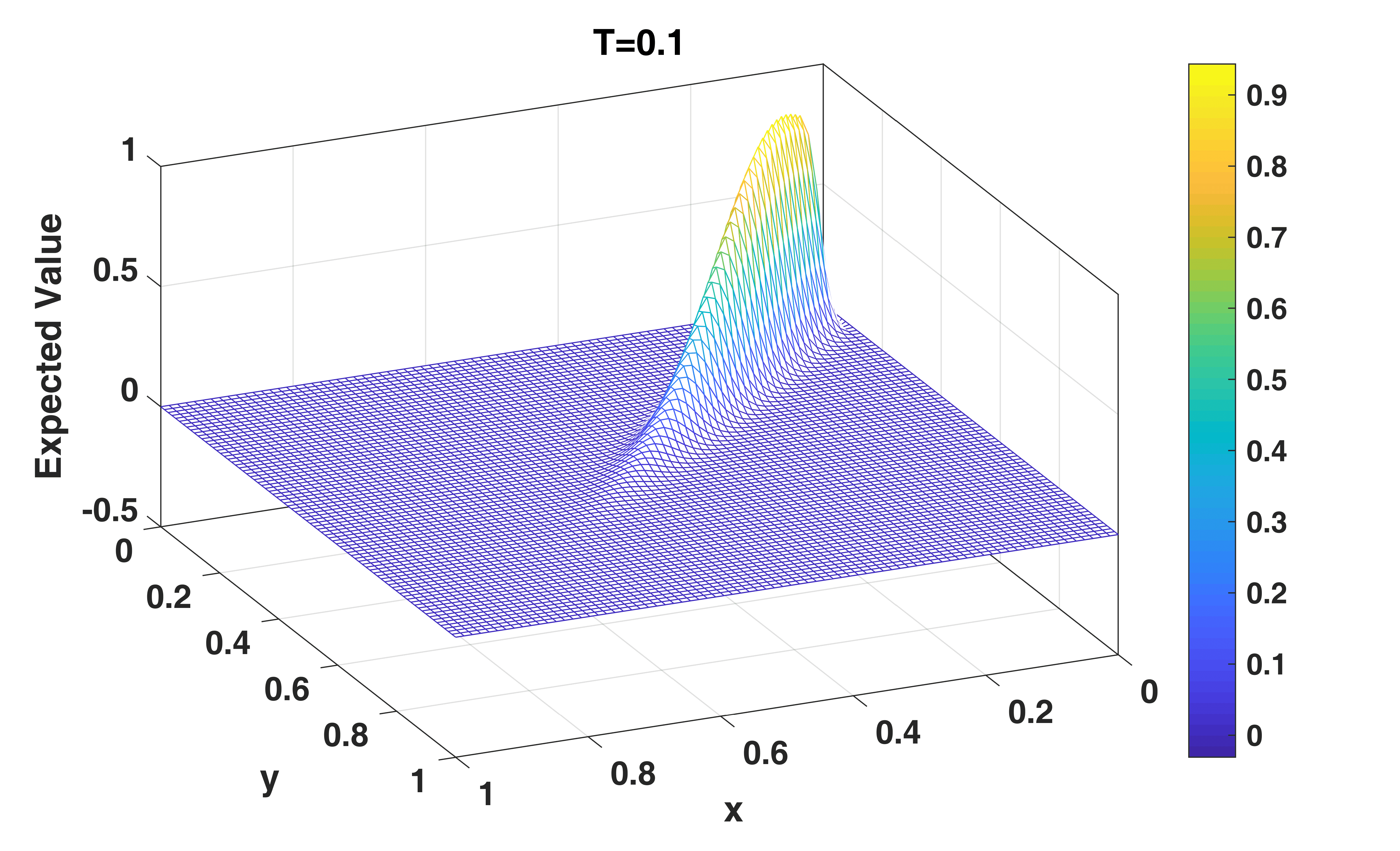}
		\includegraphics[width=8.5cm, height=7cm]{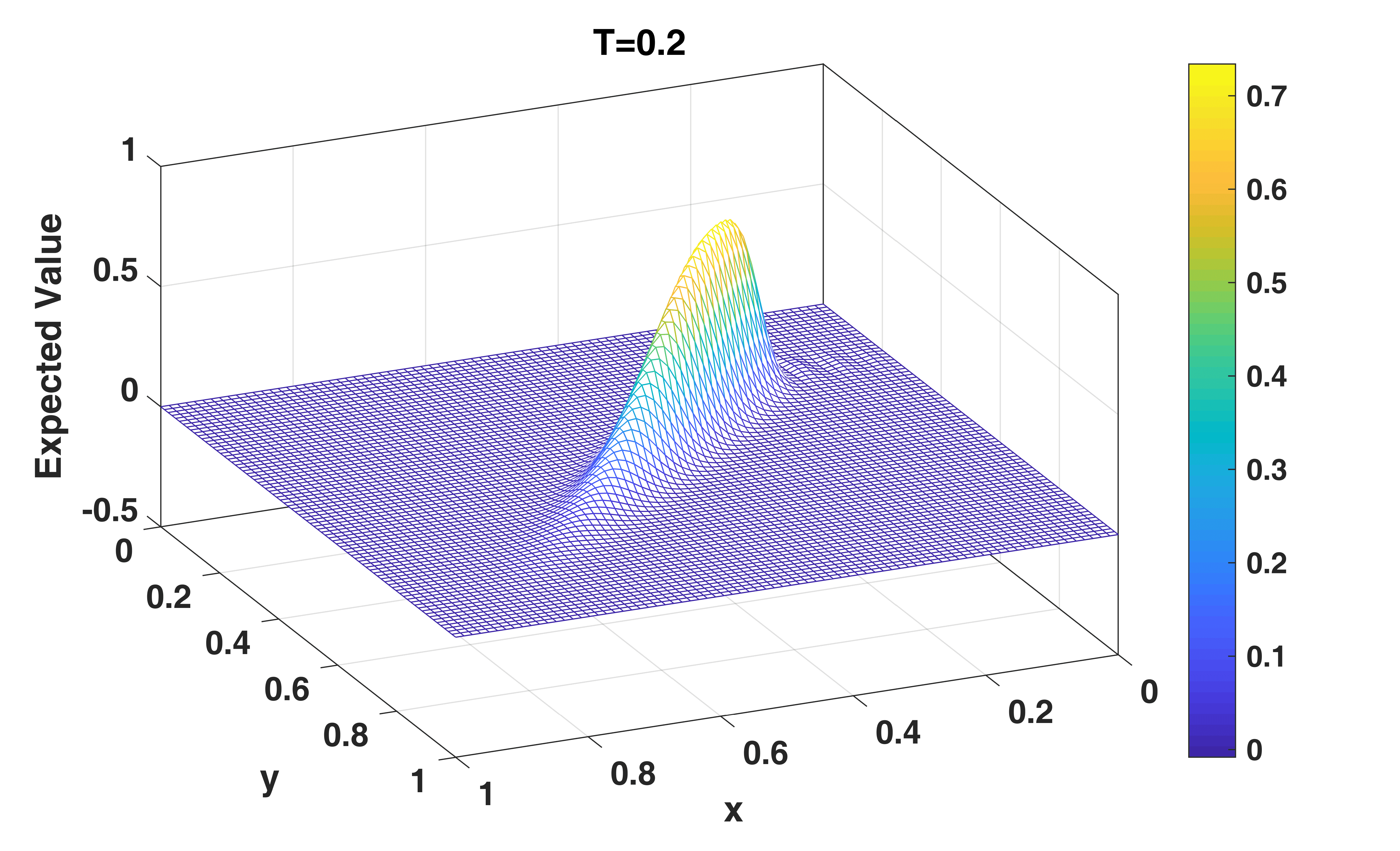}
		\includegraphics[width=8.5cm, height=7cm]{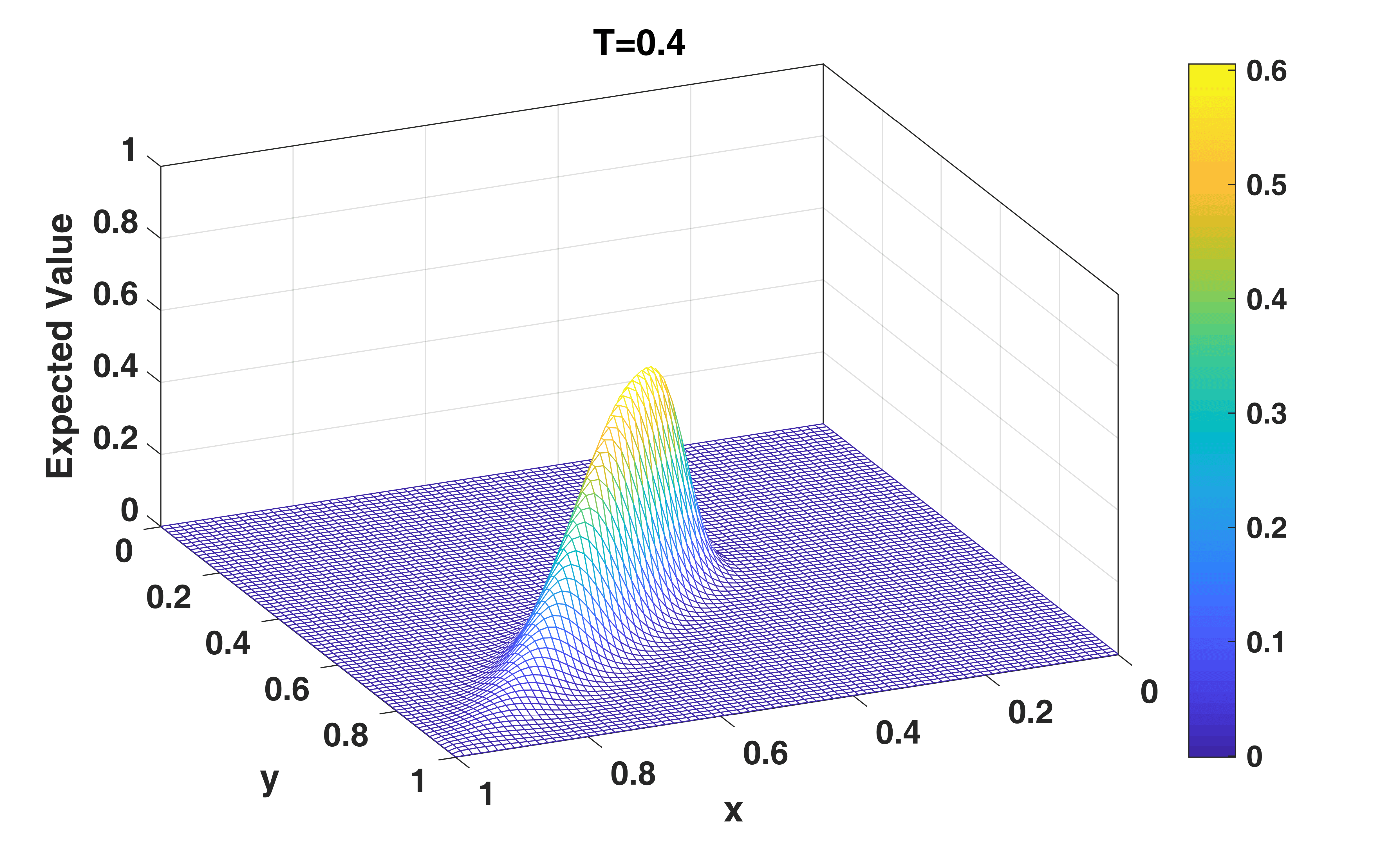}
		\includegraphics[width=8.5cm, height=7cm]{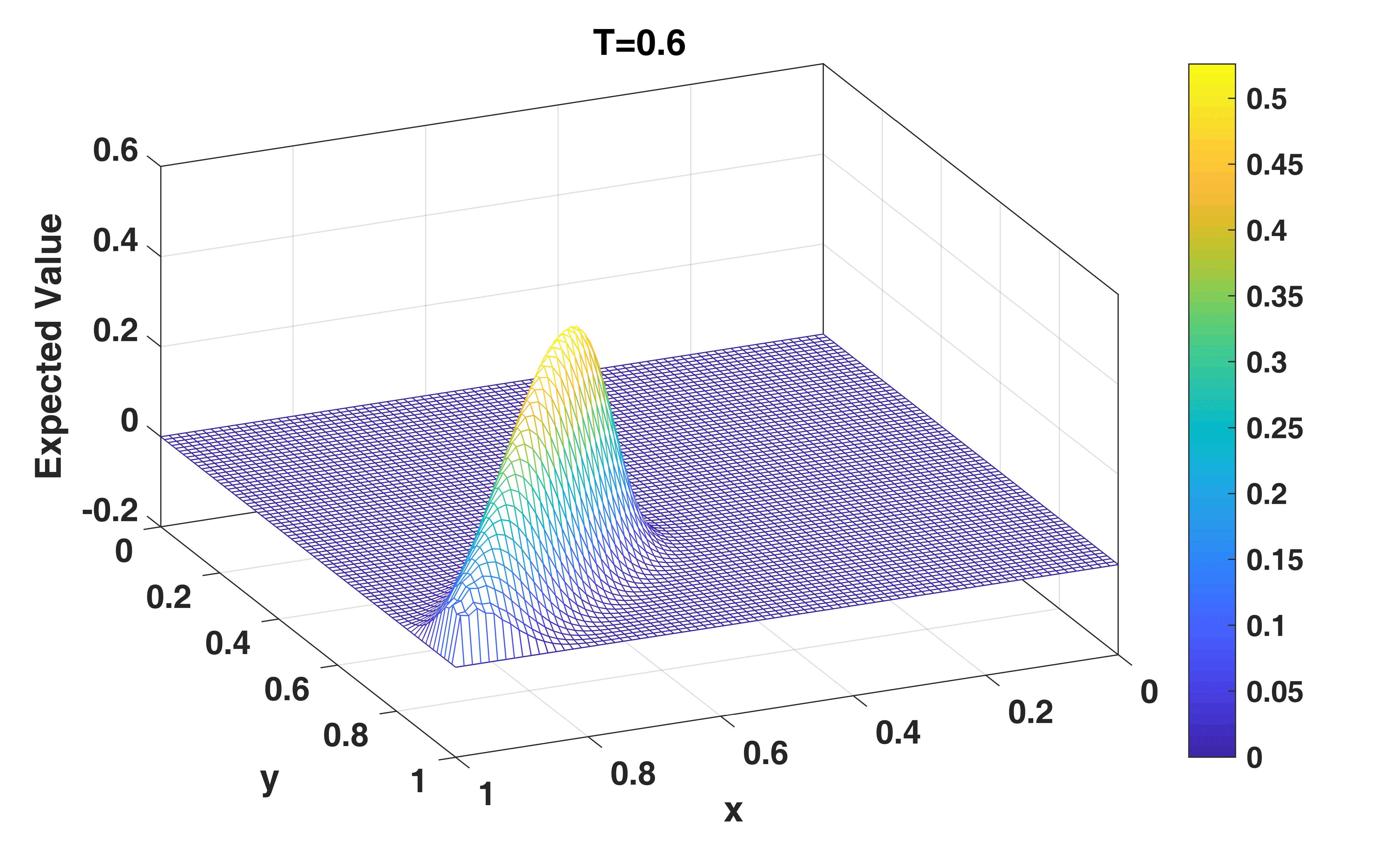}
		\includegraphics[width=8.5cm, height=7cm]{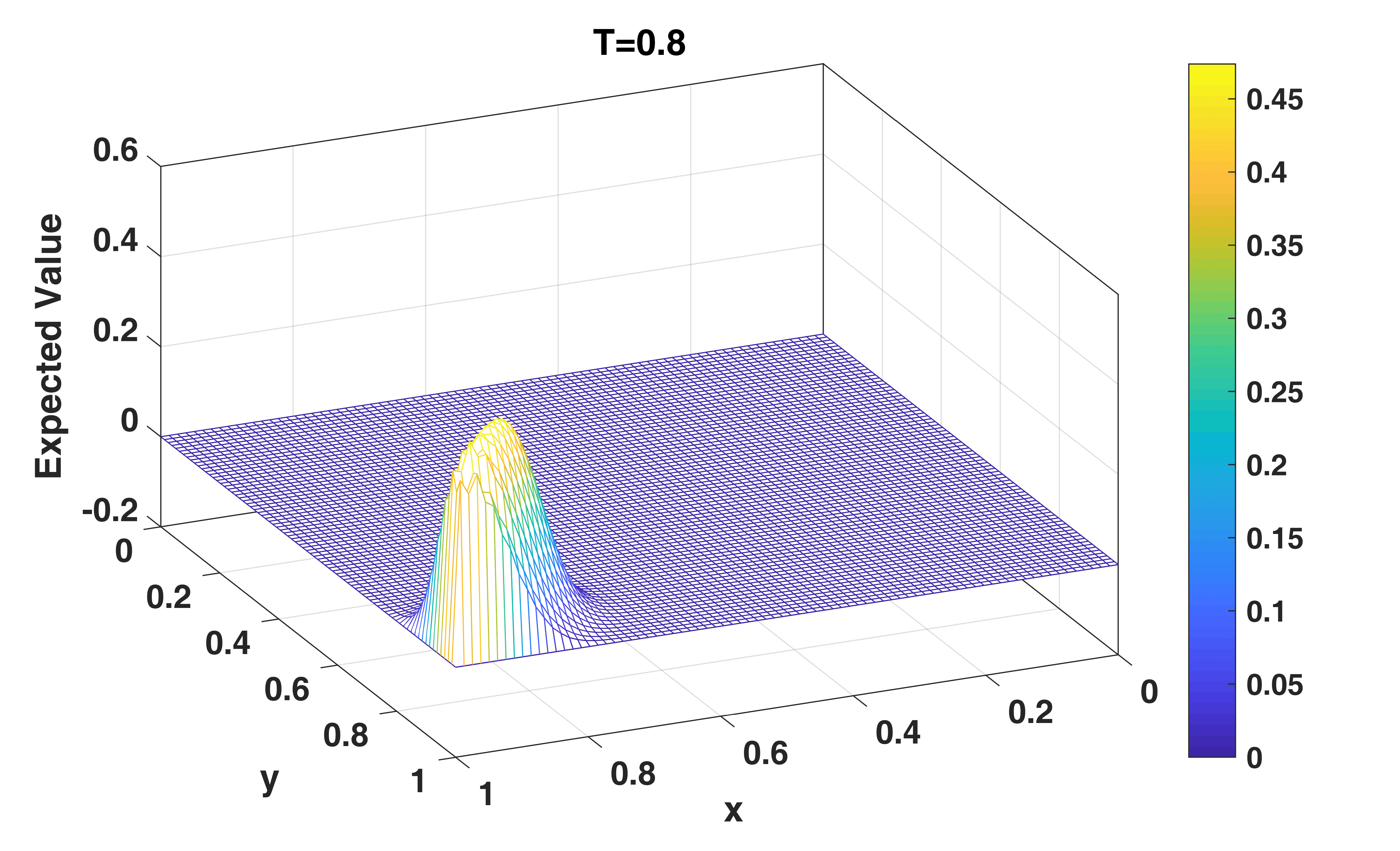}
		\includegraphics[width=8.5cm, height=7cm]{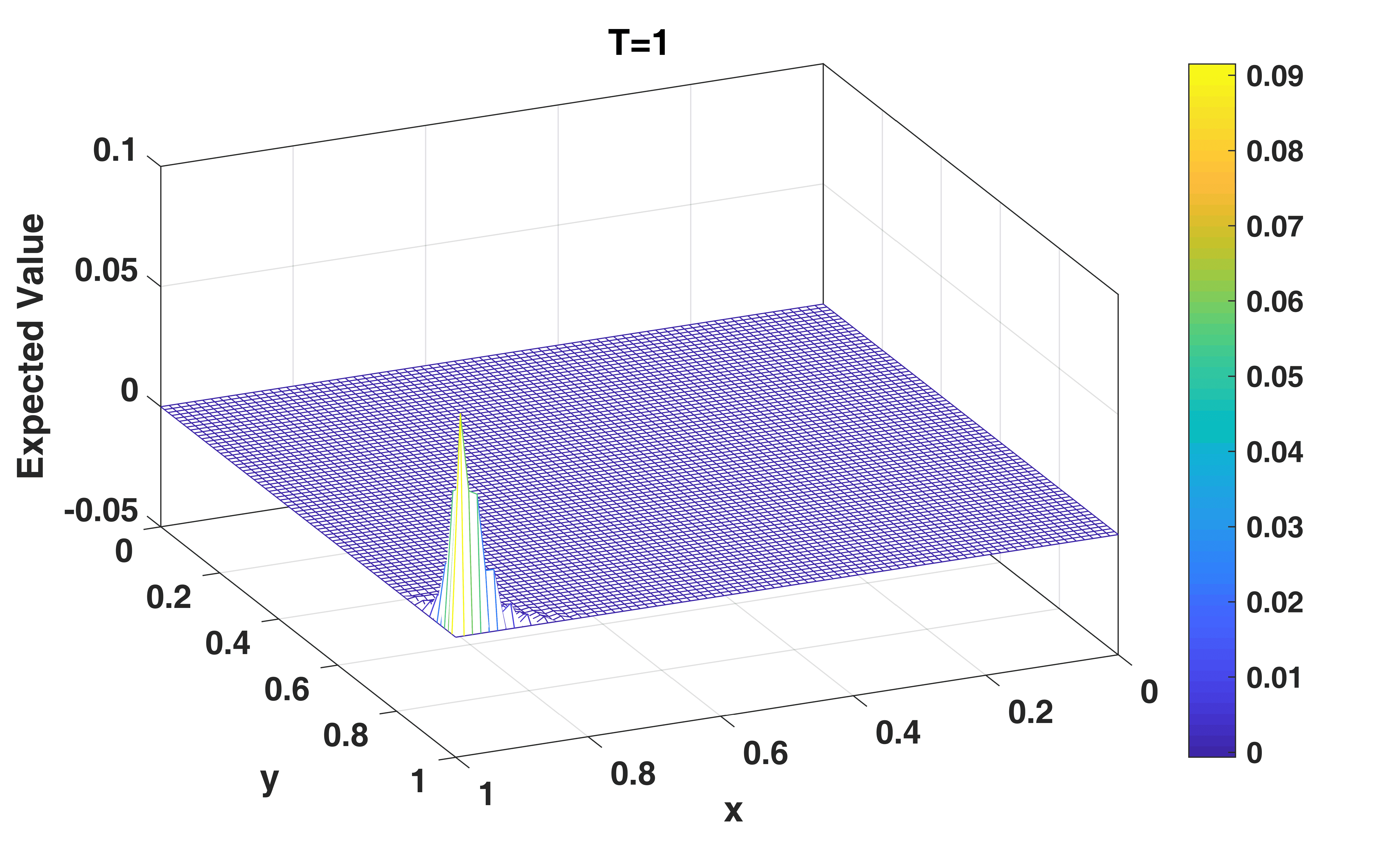}
		\caption{ The evolution of the solution (here $u$) for test problem 2.}\label{Fig3}
	\end{center}
\end{figure}
\begin{center}
	\begin{tabular}{llllllllllll}
		\multicolumn{12}{c}{\footnotesize{\textbf{Table 1}}}\\
		\multicolumn{12}{c}{\footnotesize{\textbf{Errors and computational orders obtained for Test problem 1}}}\\
		\hline\hline
		&\footnotesize{}&&\footnotesize{$N=10$}&&\footnotesize{}&&\footnotesize{$N=20$}&&\footnotesize{}&&\footnotesize{}\\
		\cline{4-6}\cline{8-10}
		&\footnotesize{$tau$}&&\footnotesize{$L_{\infty}$}&&\footnotesize{$C_1$-order}&&\footnotesize{$L_{\infty}$}&&\footnotesize{$C_1$-order}&&\footnotesize{CPU time(s)}\\
		\hline\hline
		&\footnotesize{$1/32$}&&\footnotesize{$1.2863\times10^{-3}$}&&\footnotesize{$-$}&&\footnotesize{$1.2902\times10^{-3}$}&&\footnotesize{$-$}&&\footnotesize{$0.25$}\\
		&\footnotesize{$1/64$}&&\footnotesize{$3.3204\times10^{-4}$}&&\footnotesize{$1.9537$}&&\footnotesize{$3.3304\times10^{-4}$}&&\footnotesize{$1.9538$}&&\footnotesize{$0.39$}\\
		&\footnotesize{$1/128$}&&\footnotesize{$8.3715\times10^{-5}$}&&\footnotesize{$1.9877$}&&\footnotesize{$8.3968\times10^{-5}$}&&\footnotesize{$1.9877$}&&\footnotesize{$1.5$}\\
		&\footnotesize{$1/256$}&&\footnotesize{$5.2520\times10^{-6}$}&&\footnotesize{$1.9976$}&&\footnotesize{$2.1037\times10^{-5}$}&&\footnotesize{$1.9969$}&&\footnotesize{$34$}\\
		&\footnotesize{$1/512$}&&\footnotesize{$1.3150\times10^{-6}$}&&\footnotesize{$1.9978$}&&\footnotesize{$5.2610\times10^{-6}$}&&\footnotesize{$1.9995$}&&\footnotesize{$65$}\\
		\hline\hline
	\end{tabular}
\end{center}
Table 1 and Figure \ref{Fig10} confirm the theoretical results as the computational convergence order of the proposed scheme is closed to the theoretical convergence order. 
\subsection{Test problem 2}\label{E2}
In this second numerical example, we consider a sophisticated example. The initial conditions for the considered example are based on the delta function and $\Omega=[0,1]\times [0,1]$.
 In fact, since the delta function is a discontinuous function, the initial condition is not smooth. We solve this case of groundwater model \cite{W.Liu}  using the proposed numerical procedure.
We investigate the following model
\begin{equation}\label{Exa2}
\left\{ {\begin{array}{*{20}{l}}
	{\displaystyle\frac{{\partial u}}{{\partial t}} + \mu \frac{{\partial u}}{{\partial x}} - D\frac{{{\partial ^2}u}}{{\partial {x^2}}} + \mu \frac{{\partial u}}{{\partial y}} - D\frac{{{\partial ^2}u}}{{\partial {y^2}}} + 0.6{\varpi _p}\frac{{uv}}{{\left( {1 + u} \right)\left( {v + 2} \right)}} =  dW,}\\
	{}\\
	{\displaystyle\frac{{\partial v}}{{\partial t}} + \mu \frac{{\partial v}}{{\partial x}} - D\frac{{{\partial ^2}v}}{{\partial {x^2}}} + \mu \frac{{\partial v}}{{\partial y}} - D\frac{{{\partial ^2}v}}{{\partial {y^2}}} + 0.6{\varpi _p}\frac{{uv}}{{\left( {1 + u} \right)\left( {v + 2} \right)}} =  dW,}\\
	{}\\
	{\displaystyle\frac{{\partial w}}{{\partial t}} + \mu \frac{{\partial w}}{{\partial x}} - D\frac{{{\partial ^2}w}}{{\partial {x^2}}} + \mu \frac{{\partial w}}{{\partial y}} - D\frac{{{\partial ^2}w}}{{\partial {y^2}}} + 0.6{\varpi _p}\frac{{uv}}{{\left( {1 + u} \right)\left( {v + 2} \right)}} + 2w =  dW.}
	\end{array}} \right.
\end{equation}
 In this advection-diffusion equation,  the advection coefficients are $\mu=[1,1]$ the diffusion coefficient is $D=10^{-4}$, and zero Dirichlet boundary conditions are applied. The groundwater model is a system of nonlinear equations that it explains how to remove pollutants of groundwater \cite{W.Liu}.  Now, we consider two initial conditions that they are near to the real world problems as
\begin{equation}\label{Exa3}
u(x,y,0) = v(x,y,0) = w(x,y,0) = x(1 - x)y(1 - y),
\end{equation}
and
\begin{equation}\label{Exa4}
u(x,y,0) = {\mkern 1mu} v(x,y,0) = w(x,y,0) = \delta (0,0).
\end{equation}
Relations \eqref{Exa2} and \eqref{Exa3} are respectively smooth and nonsmooth initial data.  We apply the developed technique to approximate the solution and the physical phenomena (here $u$) using the nonsmooth initial condition. Figure \ref{Fig3} illustrates the expected value of the solution  for the second test problem during different computational time.

\section{Conclusion}
\label{Section5}
The current article presents a new Legendre spectral element technique for solving the   stochastic nonlinear system of advection-reaction-diffusion equations. The main advantage of the proposed numerical procedure is that the derived mass and diffuse matrices have tridiagonal and diagonal forms, respectively.  We proved that the full-discrete scheme is unconditionally stable and convergent.  The computational results
confirm the capability of the present scheme and the theoretical concepts in  our investigation.

\bibliographystyle{elsarticle-num}
\bibliography{Ref}

\end{document}